\RequirePackage[orthodox]{nag} 
\documentclass[english, a4paper, hidelinks]{scrartcl} 
\usepackage[utf8]{inputenc} 
\usepackage[T1]{fontenc}  
\usepackage{lmodern} 
\usepackage{pdfpages}
\usepackage{scrhack}
\usepackage[main = english, ngerman]{babel} 
\usepackage{csquotes}
\usepackage{amsmath,amssymb,amsthm,bm}  
\usepackage{mathtools}
\usepackage{blkarray}
\usepackage{tabularx}
\usepackage{booktabs} 
\usepackage{enumitem}
\usepackage{varwidth}
\usepackage{siunitx}  
\usepackage{setspace}
\usepackage{graphicx}
\usepackage{wrapfig}
\usepackage{subcaption}
\usepackage{hyperref}  

\usepackage{bookmark}
\usepackage[noabbrev, capitalise]{cleveref}

\usepackage{xcolor}
\usepackage{grffile}
\usepackage{float} 
\usepackage{tikz, pgf}
\usetikzlibrary{fit,decorations.pathreplacing,patterns,shapes}
\usetikzlibrary{plotmarks}
\usetikzlibrary{arrows.meta, automata}
\usetikzlibrary{calc,intersections}
\usetikzlibrary{external}
\usepackage{multirow}
\usepackage{multicol} 
\usepackage{pgfplots}
\pgfplotsset{compat=newest}
\usepgfplotslibrary{patchplots}

\usepackage{url}
\usepackage{verbatim}
\usepackage[usestackEOL]{stackengine}

\usepackage[color=orange!80, textsize=tiny]{todonotes}

\theoremstyle{plain}
\newtheorem{theorem}{Theorem}[section] 
\newtheorem*{theorem*}{Theorem}
\newtheorem{lemma}[theorem]{Lemma} 
\newtheorem*{lemma*}{Lemma} 
\newtheorem{korollar}[theorem]{Corollary} 
\newtheorem*{korollar*}{Corollary} 
 
\newtheorem*{satz*}{Satz} 
\theoremstyle{definition}
\newtheorem{definition}[theorem]{Definition}
\newtheorem*{definition*}{Definition}

\newtheorem*{example*}{Example}

\newtheorem{assumption}{Assumption}
\newtheorem*{assumption*}{Assumption}
 
\theoremstyle{remark} 
\newtheorem{remark}[theorem]{Remark}
\newtheorem*{remark*}{Remark}

\newtheorem*{anmerkung*}{Anmerkung}

\newcommand{\R}{\mathbb{R}}

\newcommand{\G}{\mathcal{G}}

\renewcommand{\P}{\mathcal{P}}
\newcommand{\pipe}{\mathrm{pipe}}
\newcommand{\V}{\mathcal{V}}
\newcommand{\E}{\mathcal{E}}
\renewcommand{\H}{\mathcal{H}}
\newcommand{\C}{\mathrm{comp}}

\newcommand{\dd}{\, \mathrm{d}}

\newcommand{\cl}{{\ell}}
\newcommand{\crr}{{\mathrm{r}}}

\newcommand{\mol}{{\mathrm{mol}}}

\newcommand{\cycle}{{\mathrm{cycle}}}
\newcommand{\sol}{{\mathrm{sol}}}

\newcommand{\mach}{\mathrm{Ma}}

\newcommand{\sign}{{\mathrm{sign}}}

\newcommand{\dx}{\frac{\partial}{\partial x}}

\newcommand{\rhong}{\rho_{\mathrm{NG}}}
\newcommand{\rhohh}{\rho_{\mathrm{H}_2}}
\newcommand{\qng}{q_{\mathrm{NG}}}
\newcommand{\qhh}{q_{\mathrm{H}_2}}
\newcommand{\png}{p_{\mathrm{NG}}}
\newcommand{\phh}{p_{\mathrm{H}_2}}

\newcommand{\fr}{{\mathrm{fr}}}

\DeclareMathOperator*{\argmin}{arg\,min}

\definecolor{rwth}   {RGB}{  0  84 159}
\definecolor{rwth-75}{RGB}{ 64 127 183}
\definecolor{rwth-50}{RGB}{142 186 229}
\definecolor{rwth-25}{RGB}{199 221 242}
\definecolor{rwth-10}{RGB}{232 241 250}

\definecolor{rot}   {RGB}{204   7  30}
\definecolor{orange}   {RGB}{246 168   0}
\definecolor{lila}   {RGB}{122 111 172}
\definecolor{bordeaux-50}{RGB}{205 139 135}

\usepackage[backend=biber,style=numeric,giveninits=true,hyperref=true,isbn=false,url=true,doi=true,maxnames=10]{biblatex}
\usepackage{a4wide}
\numberwithin{equation}{section} 
\title{Logbook}
\subtitle{Existence of Steady States to the Mixture Model (again)}
\author{Alena Ulke}
\date{\today}
\begin{document}
%
%
%
\begin{center}
    \textbf{\textrm{\Large On the Existence of Steady States for Blended Gas Flow with Non-Constant Compressibility Factor on Networks}} \\[10pt]
    S. G\"ottlich\footnotemark[1], M. Schuster\footnotemark[2], A. Ulke\footnotemark[1]
\end{center}
\footnotetext[1]{University of Mannheim, Department of Mathematics, 68131 Mannheim, Germany (ulke@uni-mannheim.de, goettlich@uni-mannheim.de)}
\footnotetext[2]{Friedrich-Alexander University Erlangen-Nürnberg, Department of Mathematics, 91058 Erlangen, Germany (michi.schuster@fau.de)}


\paragraph*{Abstract}
In this paper, we study hydrogen-natural gas mixtures transported through pipeline networks. The flow is modeled by the isothermal Euler equations with a pressure law involving a non-constant, composition-dependent compressibility factor. For a broad class of such compressibility models, we prove the existence of steady-state solutions on networks containing compressor stations. The analysis is based on an implicit representation of the pressure profiles and a continuity argument that overcomes the discontinuous dependence of the gas composition on the flow direction. Numerical examples illustrate the influence of different compressibility models on the resulting states.
%
\section{Introduction and Motivation}

The ongoing transformation of the global energy system towards climate neutrality has placed hydrogen at the center of scientific and industrial discussions. As an energy carrier, hydrogen possesses several highly attractive properties. 
It can be produced from renewable energy sources such as wind or solar power, and its use does not generate carbon dioxide emissions, releasing only water and heat \cite{Bundesregierung2024, EuropeanCommission2020}. 
Moreover, hydrogen is regarded as a key enabler for the decarbonization of sectors that cannot be easily electrified, including parts of heavy industry and long-distance transport. In addition, hydrogen provides flexibility within the energy system by enabling large-scale energy storage and sector coupling \cite{EuropeanCommission2020, USEnergy2017}, which underlines hydrogen’s anticipated relevance in future energy infrastructures.  

Despite its promising potential, hydrogen will not immediately replace natural gas as a primary energy source. Instead, a gradual integration into existing infrastructures is currently considered the most feasible pathway. In this context, blending hydrogen into natural gas pipeline networks has emerged as a pragmatic transition strategy, allowing for a reduction in greenhouse gas emissions while relying on well-established transport systems. Large-scale hydrogen transport networks are already under development worldwide, cf. \cite{HydrogenEU2022} for Europe, \cite{HydrogenASIA2023} for Asia, and \cite{HydrogenUSA2023} for the United States. Nevertheless, for the foreseeable future, hydrogen-natural gas mixtures are expected to dominate practical applications rather than pure hydrogen flows.

The safe and efficient operation of pipeline networks carrying hydrogen-enriched natural gas requires a detailed mathematical understanding of the underlying flow processes. While the transport of pure natural gas (or any other single gas) can be accurately described by the isothermal Euler equations, cf. \cite{BandaHertyKlar2006,Gas-Modell:DomschkeHillerLangTischendorf2017,GugatUlbrich2018,KochEtAl2015}:
\begin{subequations}
    \label{eq:iso1}
	\begin{align}
		\partial_t \rhong(t,x) + \partial_x \qng(t,x) &= 0, \label{eq:isoMassConservation} \\
		\partial_t \qng(t,x) + \partial_x \bigg[ \png(\rhong(t,x)) + \frac{\qng^2(t,x)}{\rhong(t,x)} \bigg] &= - \frac{\lambda_{\text{fr}}}{2D} \frac{\qng(t,x) \vert \qng(t,x) \vert}{\rhong(t,x)}, \label{eq:isoMomentumBalance}
	\end{align}
\end{subequations}
the transport of a two-gas-mixture can be modeled by mass conservation \eqref{eq:isoMassConservation} for each gas and momentum balance \eqref{eq:isoMomentumBalance} for the mixing, cf. \cite{BotheDreyer2015, Kazi2024, MalekSoucek2022}, i.e., 
\begin{subequations}\label{eq:dynamicEquation}
    \begin{align}
        \partial_t \, \rhong(t,x) + \partial_x \, \qng(t,x) &= 0, \\
        \partial_t \, \rhohh(t,x) + \partial_x \, \qhh(t,x) &= 0, \\
        \partial_t \, q(t,x) + \partial_x \left[ p(\rhong(t,x), \rhohh(t,x)) + \frac{q^2(t,x)}{\rho(t,x)} \right] 
            & = - \frac{\lambda_\fr}{2D} \frac{q(t,x) \vert q(t,x) \vert}{\rho(t,x)}.
    \end{align} 
\end{subequations}
The time-space domain is $(t,x) \in [0,t_{\mathrm{final}}] \times [0,L]$. Furthermore, $\rhong, \rhohh$ describe the densities and $\qng, \qhh$ the flow of natural gas and hydrogen, $\rho = \rhong + \rhohh$ and $q = \qng + \qhh$ are the density and the flow of the mixture, respectively. Further, $D > 0$ is the pipe diameter and $\lambda_{\fr}$ the pipe friction coefficient. The pressure-density relation of the gas is described by an appropriate state equation or pressure law. For pure natural gas, the pressure is commonly expressed as
\begin{equation*}
	\png(\rhong) = \frac{R \, T}{M_{\text{NG}}} Z(\png) \rhong,
\end{equation*}
where $R$ denotes the universal gas constant, $M_{\text{NG}}$ the molar mass of natural gas, $T$ the constant gas temperature, and $Z(\png)$ the compressibility factor. 

For natural gas transport, the modeling of the compressibility factor is well established in the literature. Common choices include a constant approximation corresponding to the \textit{ideal gas law}, cf. \cite{DomschkeKolbLang2022,Gugat2024,SchusterDiss2021}, linear approximations such as those proposed by the \textit{American Gas Association}, cf. \cite{Gugat2018,Schmidt2016}, and quadratic models based on \textit{Papay’s formula}, cf.  \cite{Schmidt2016,WaltherHiller2017}. These approximations are widely used in both stationary and transient gas network models and have proven to yield reliable results for pure natural gas flow. 
However, for gas mixtures, well established choices are still missing. While the molar mass of a mixture is typically defined as a convex combination of molar masses of the individual constituents, no universally accepted model exists for the corresponding compressibility factor.
Several approaches have been proposed in the literature, primarily in the context of numerical simulations. 
For instance, \cite{Pfetsch2024,Ulke2025} approximates the compressibility factor of the mixture by a convex combination of constant compressibility factors, whereas \cite{Brodskyi2025} employs a convex combination of linear compressibility models. 
Further, \cite{Bermudez2022} adopts a different strategy by combining critical temperatures and molar masses in order to apply a quadratic compressibility factor. 

Although these approaches are all inspired by well-established compressibility models for pure natural gas, their extension to hydrogen-natural gas mixtures is not unique. Different modeling choices lead to distinct pressure-density relations and, consequently, to noticeably different steady states, even under identical boundary conditions, as shown in \cref{fig:simulationCompressibilityFactor}. For the constant compressibility factor \cite{Pfetsch2024,Ulke2025}, we apply
\begin{equation} \label{eq:compressibilityConstant}
    Z(\png) = Z(\phh) = 1,
\end{equation}
and for the linear model \cite{Brodskyi2025,Schmidt2016}, we use
\begin{equation} \label{eq:compressibilityLinear}
    Z(p) = 1 + \big( \eta\ \alpha_{\text{H}_2} + (1 - \eta) \alpha_{\text{NG}} \big) p \qquad \text{with} \qquad \alpha_X = \frac{0.257}{p_{c,X}} - \frac{0.5333\ T_{c,X}}{p_{c,X}\ T},
\end{equation}
where $X \in \{ \text{H}_2, \text{NG} \}$, $p$ is the pressure of the mixture and $\eta$ is the percentage of hydrogen in the mixture. Further, $T$ is the gas temperature and $p_{c,X}, T_{c,X}$ are the critical pressure and temperature of the gases. The quadratic model \cite{Schmidt2016,Bermudez2022} is given by
\begin{equation} \label{eq:compressibilityQuadratic}
    Z(p) = 1 - 3.52 \exp\Big( -2.26 \frac{T}{T_c} \Big) \frac{p}{p_c} + 0.274 \exp\Big( -1.878 \frac{T}{T_c} \Big) \frac{p^2}{p^2_c},
\end{equation}
where $p_c, T_c$ are the critical pressure and temperature of the mixture, computed by the convex combination of the critical pressure and temperature of the gases. The simulation results in \cref{fig:simulationCompressibilityFactor} demonstrate two effects: first, a clear difference between constant and non-constant compressibility models, with the non-constant models nearly coinciding for pure natural gas, and second, a growing discrepancy between the models as the hydrogen percentage increases. The lack of a canonical compressibility model for gas mixtures indicates that analytical results should not rely on a specific functional form of the compressibility factor. 

\begin{figure}[htbp]
	\centering
   \includegraphics[width=\textwidth, trim=50mm 0mm 40mm 0mm, clip]{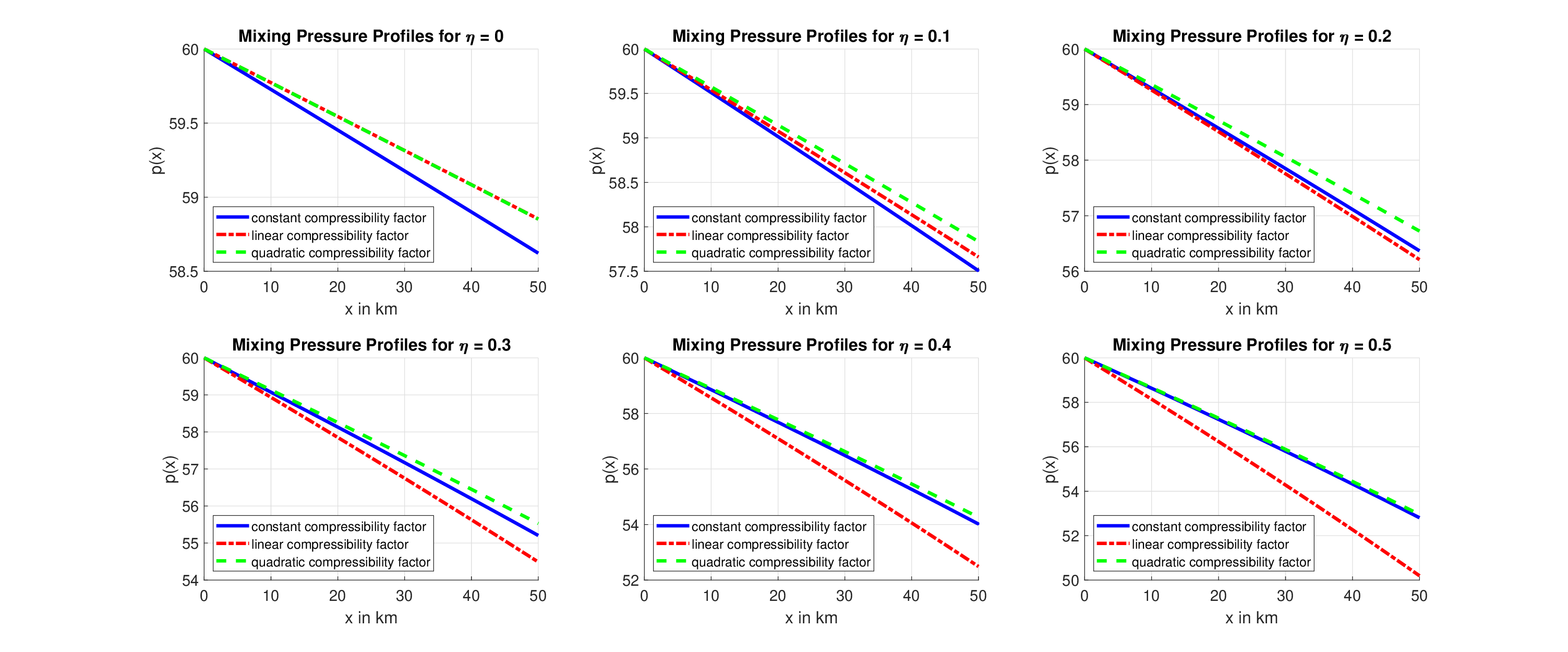}
	\caption{Simulation results for the compressibility factor models \eqref{eq:compressibilityConstant} - \eqref{eq:compressibilityQuadratic} on a single pipe with $L = 50$ km, $\lambda_\fr = 0.05$ and $D = 0.5$ m. The gas parameters are given by $R = 8.3145$ J/(mol K)$^{-1}$, $T = 283.15$ K, $p_{c,\text{H}_2} = 13.15$ bar, $p_{c,\text{NG}} = 46.01$ bar, $T_{c,\text{H}_2} = 33.19$ K and $T_{c,\text{NG}} = 204.62$ K.
    }
	\label{fig:simulationCompressibilityFactor}
\end{figure}

For the mixing model with a constant compressibility factor, the existence of steady states has been shown in \cite{Ulke2025} for pure transport networks.
In contrast, the existence of steady-state solutions for the mixture model with non-constant compressibility factor and networks containing controllable elements such as compressor stations remains an open question.
This gap motivates the present work, in which we establish the existence of steady states for a broad class of compressibility factors and extend the analysis to networks with compressor stations.
%
The existence proof in \cite{Ulke2025} heavily utilizes that the pressure change along a path depends continuously on the boundary data, despite the discontinuous dependence of the composition on the flow direction.
The key observation to establish the continuous dependence on the boundary data is that
in the steady-state case, \cref{eq:dynamicEquation} reduces to an ordinary differential equation (ODE) for the pressure profile.
For ideal gas, this ODE admits a closed-form solution allowing us to analyze the influence of the discontinuities in composition on the pressure change explicitly.
For arbitrary compressibility factors, however, such a closed-form solution is unavailable in general. Thus, our key contribution is addressing the resulting technical difficulties. 

\medbreak
This article is structured as follows: In Section 2, we introduce the gas flow model for mixtures on networks and derive an implicit solution of the ordinary differential equation arising from steady-state version of \cref{eq:dynamicEquation} for arbitrary compressibility factors.
Afterwards, we establish the existence of steady-state solutions for tree-shaped networks and networks with a single cycle in Section 3. 
Finally, in Section 4, we conclude with a numerical study showcasing the influence of different choices of the compressibility factor \(Z\) on the steady-state solution of the gas flow model for mixtures.

\section{Modeling Gas Flow of Mixtures on Networks}
Gas networks involve different elements, e.g., producers and consumers that deliver gas to and withdraw gas from the network, pipes to transport gas, and compressors to regulate pressure levels in the network. 
In the following, we introduce essential terminology and notation before we delve into the modeling of gas flow. 

Let a directed, connected graph \( \G = (\V, \E)\) with the set of nodes \(\mathcal{V}\), which represent network junctions and serve as gas entry or exit points, and the set of edges \(\mathcal{E} \subseteq \mathcal{V} \times \mathcal{V}\) be given. Each edge \(e \in \mathcal{E}\) either represents a pipe of length \(L_e > 0\) or a compressor station with compressor ratio \( \gamma_e \geq 1\), and we accordingly partition the edge set into the two disjoint subsets \(\E_\pipe \subseteq\E\) and \( \E_\C \subseteq \E\) of pipes and compressors, respectively. For an edge \(e = (v,w) \in \mathcal{E}\), let \(f(e) = v\) and \(h(e) = w\) denote the \emph{start} of \emph{foot node} and the \emph{end} or \emph{head node}, respectively. 
The set of nodes is given by the two disjoint sets \( \V_{<0}\), representing the \emph{supply nodes}, where gas is injected into the network, and \( \V_{\ge0}\), representing the \emph{demand nodes}, where mass is conserved or gas is withdrawn from the network. 
For each node, we denote the set of \emph{incoming} and \emph{outgoing edges} by
\begin{equation*}
	\E_-(v) = \{e \in \E \mid e = (\bullet, v)\}
	\quad \text{and} \quad
	\E_+(v) = \{e \in \E \mid e = (v, \bullet)\},
\end{equation*}
respectively. The set \( \E(v) = \E_-(v) \cup \E_+(v)\) contains all edges incident to \(v\). Further, we require the \emph{node edge incidence matrix} \( A \in \R^{\vert \V\vert \times \vert \E \vert}\) with entries:
\begin{equation}\label{eq: incidence matrix}
	a(v,e) = \begin{cases}
		-1, & \text{if } v = f(e), \\
		1, & \text{if } v = h(e), \\
		0, & \text{otherwise}.
	\end{cases}
\end{equation}
In order to describe gas flow on a network, we split the dynamic into three: the flow along pipes, the flow across nodes, and the operation of compressors.

\subsection{Gas Flow through Pipes}

The gas flow in pipeline networks is governed by the \textit{isothermal Euler equations} \eqref{eq:iso1}, a system of hyperbolic partial differential equations (PDEs). Isothermality implies constant gas temperature over space and time. Since we are interested in the flow of mixtures consisting of hydrogen and natural gas, we consider a framework allowing for mixtures composed of two arbitrary constituents \( i \in \{1,2\}\). 
We assume perfect mixing, i.e., both constituents have the same velocity and temperature.  Within this framework, we can model the flow of hydrogen and natural gas by setting hydrogen as constituent \( i = 1\) and natural gas as constituent \( i= 2\), cf. \eqref{eq:dynamicEquation}. Let $\eta_e$ be the mass fraction of the first constituent (i.e., hydrogen) in the mixture on edge $e \in \mathcal{E}$. Thus, \eqref{eq:dynamicEquation} is equivalent to the system:
\begin{subequations} \label{eq: Euler}
	\begin{align}
		\partial_t \rho_e + \partial_x q_e & = 0, \\
		\partial_t\left(\eta_e \rho_e\right) + \partial_x \left(\eta_e q_e\right) & = 0, 
		\vphantom{\left[ p_e + \frac{q_e^2}{\rho_e} \right]}\\
		\partial_t q_e + \partial_x \left[ p_e + \frac{q_e^2}{\rho_e} \right] & = - \frac{\lambda_\fr}{2 D_e} \frac{q_e \vert q_e \vert}{\rho_e}, \label{eq: Euler pressure loss}
	\end{align}
\end{subequations}
where \(\rho_e\) denotes the gas density, \( q_e\) the mass flow, \( p_e\) the pressure of the gas mixture. The constants \( D_e\) and \( \lambda_\fr\) denote the diameter and friction factor of the pipe, respectively.
The pressure and density are connected by the state equation
\begin{equation*}
	p_e(\rho_e) = \frac{R \, T}{M} Z(p_e) \rho_e,
\end{equation*}
where \( R\) is the universal gas constant, \( T\) the temperature, \( M\) the molar mass, and \(Z(p_e) > 0\) the compressibility factor of the gas. 
When analyzing the flow of gas mixtures, it is important to note that variations in the composition result in varying gas properties, including the molar mass and compressibility factor.
However, since the gas mixture is always composed of the same fixed constituents, we can express the gas properties of a mixture based on the composition and properties of its constituents, i.e., the molar mass \( M \) and the compressibility factor \( Z\) also depend on the gas composition.
For the remainder of this section, we omit the subscript \(e\) to improve the readability. 

We determine the molar mass of a gas mixture based on the molar masses \( M_i\) of its constituents, along with the \emph{molar fraction} \(\eta_\mol\) of the first constituent, cf.  \cite{Bermudez2022,Pfetsch2024}. The molar mass \( M\) of the mixture is then given by
\begin{equation}\label{eq: molar mass in molar fraction}
	M = \eta_\mol M_1  + (1 - \eta_\mol) M_2.
\end{equation}
We can also express the molar mass \( M\) in terms of mass fraction \( \eta\) by using the relation:
\begin{equation}\label{eq: mass fraction to molar fraction}
	\eta_\mol = \frac{n_1}{n_1 + n_2} \quad \text{with} \quad 
	n_1 = \frac{\eta}{M_1}, \quad n_2 = \frac{1-\eta}{M_2},
\end{equation}
where \( n_i\) is the number of moles of constituent \(i\) in the mixture. We substitute the molar fraction \( \eta_\mol\) in \cref{eq: molar mass in molar fraction} by using the relation~\labelcref{eq: mass fraction to molar fraction}, which yields:
\begin{equation}
	M = M(\eta) = \frac{M_1 M_2}{\eta M_2 + (1- \eta) M_1}
	\quad \Leftrightarrow \quad \frac{1}{M(\eta)} = \frac{\eta}{M_1}  + \frac{(1 - \eta)}{M_2}.
\end{equation}

Next, we model dependence of the compressibility factor \( Z\) on the gas composition, which we express using the mass fraction since the relation~\labelcref{eq: mass fraction to molar fraction} allows us to convert the mass fraction \( \eta\) into the molar fraction \( \eta_\mol\) and vice versa:
\begin{align}\label{eq: conversion mass and molar fraction}
	\eta_\mol & = \frac{\eta \, M_2 }{\eta\, M_2 + (1 - \eta) \, M_1} \qquad \text{and} \qquad
	\eta = \frac{ \eta_\mol M_1}{\eta^\mol M_1 +  (1 - \eta_\mol) M_2}.
\end{align}
As there exist multiple approaches to model the compressibility factor, cf. \cite{Greg2008}, we adopt a general approach and treat the compressibility factor as a function of the pressure \(p\) and the mass fraction \( \eta\), satisfying the following assumptions:
\begin{assumption}\label{assumption}
	The compressibility factor \( Z(\eta, p)\) is integrable and differentiable with respect to \(p\) for every mass fraction \( \eta \in [0,1]\)  and pressure \( p \in [0,\infty)\).
\end{assumption}
\begin{assumption}\label{assumption 2}
	The compressibility factor \( Z(\eta, p)\) is differentiable with respect to \(\eta\) for every mass fraction \( \eta \in [0,1]\) and pressure \( p \in [0,\infty)\).
\end{assumption}
\noindent 
Finally, we obtain the pressure law for gas mixtures linking the pressure \(p\), the density \(\rho\) and the composition of the gas, which is encoded through the mass fraction \( \eta\):
\begin{equation}\label{eq: pressure law}
	p(\rho) = \frac{R \, T}{M(\eta)} \, Z(\eta, p) \, \rho
	= \left[\frac{R \, T}{M_1} \eta  + \frac{R \, T}{M_2} (1 - \eta)\right] Z(\eta, p) \, \rho.
\end{equation}


	%

\medbreak
Since we aim to analyze the steady state gas flow, the temporal derivatives in the Euler equations~\labelcref{eq: Euler} vanish. As a result, the first two equations in the isothermal Euler equations imply that the mass flow \( q\) and the mass fraction \(\eta\) are constant along the pipe, reducing the PDE system to an ordinary differential equation (ODE):
\begin{align} \label{eq: pressure ODE}
	\partial_x \left[ p(x) + \frac{q^2}{\rho(x)} \right] 
	& = -\frac{\lambda_\fr}{2 D} \frac{q \vert q \vert}{\rho(x)}
	\quad \text{for} \quad x \in [0, L].
\end{align}

In the following, we solve this ODE to obtain the pressure profile \(p(x)\) along the pipe and split the derivation of the pressure profile into two steps:
\begin{enumerate}[label=(\roman*)]
	\item We eliminate the density \( \rho\) from the ODE using the pressure law~\labelcref{eq: pressure law}, which yields an ODE depending only on the variables \(q\), \(\eta\), and \(p\). \label{item: eliminate}
	\item We solve the ODE from~\labelcref{item: eliminate} for the pressure \(p\) using separation of variables.
\end{enumerate}
We begin by transforming the ODE~\labelcref{eq: pressure ODE} into an ODE independent of the density \( \rho\). First, we consider the right side of the equation and invoke  the pressure law~\labelcref{eq: pressure law}:
\begin{equation*}
	- \frac{\lambda_\fr}{2 D} \frac{q \vert q \vert}{\rho}
	= - \frac{\lambda_\fr}{2 D} \frac{R \, T}{M(\eta)} \frac{Z(\eta, p)}{p} q \vert q \vert.
\end{equation*}
Next, we consider the left side of the equation, where we additionally invoke the chain rule, which results in:
\begin{align*}
	\partial_x \left[p + \frac{q^2}{\rho} \right]
	= \partial_x \left[p +  \frac{R \, T} {M(\eta)} q^2 \frac{Z(\eta, p)}{p} \right] 
	& = \partial_x p + \frac{R \, T}{M(\eta)} q^2 \, \partial_p \left[ \frac{Z(\eta, p)}{p} \right] \partial_x p, \\
	& = \left[ 1 + \frac{R \, T}{M(\eta)} q^2 \left(\frac{\partial_p Z(\eta, p)}{p} 
	- \frac{Z(\eta,p)}{p^2} \right) \right] \partial_x p.
\end{align*}
Recall that the mass flow \(q\) and the mass fraction \( \eta \) are constant along the pipe.
Combining both, we obtain the following formulation for the pressure loss ODE~\labelcref{eq: pressure ODE}:
\begin{align}
	&& \left[ 1 + \frac{R \, T}{M(\eta)} q^2 \left(\frac{\partial_p Z(\eta, p)}{p} 
	- \frac{Z(\eta,p)}{p^2} \right) \right] \partial_x p
	& = - \frac{\lambda_\fr}{2 D} \frac{R \, T}{M(\eta)} \frac{Z(\eta, p)}{p} q \vert q \vert, \nonumber \\
	\Leftrightarrow && \underbrace{\left[ \frac{p}{Z(\eta, p)} + \frac{R \, T }{M(\eta)} \, q^2 \, \left( 
		\frac{\partial_p Z(\eta, p)}{Z(\eta, p)} - \frac{1}{p} \right) \right]}_{ = F^{\prime}(\eta, q, p)} \partial_x p
	& = - \frac{\lambda_\fr}{2 D} \frac{R \, T}{M(\eta)} q \vert q \vert.  \label{eq: pressure ODE reduced}
\end{align}

The next step is to solve the ODE using separation of variables. Since we have already manipulated the equation into a form suitable for direct integration with respect to \(x\), we only need an anti-derivative of \(F^{\prime}\), which -- up to the integration constant -- reads:
\begin{align}\label{eq: F}
	F(\eta, q, p) & = \int \frac{p}{Z(\eta, p)} \dd p + \frac{R \, T \, q^2}{M(\eta)} \left[ \ln(\vert Z(\eta, p)\vert) - \ln (\vert p \vert) \right]. 
\end{align}
Thus, integrating the ODE~\labelcref{eq: pressure ODE reduced} with respect to \(x\) yields:
\begin{equation}\label{eq: potential formulation}
	F(\eta, q, p(x)) = F(\eta, q, p(0)) - \frac{\lambda_\fr}{2 D} \frac{R \, T}{M(\eta)} q \vert q \vert \, x \quad \text{for} \quad x \in [0,L].
\end{equation}

\begin{remark}\label{remark: choice Z}
	As there are a variety of approaches to model the compressibility factor \( Z\) of a gas (mixture), we provide a brief overview. We also determine the integral \( \int \frac{p}{Z(\eta,p)} \dd p\) for these examples, since it does not have a closed form expression for general \(Z\).
	\begin{itemize}
        \item For a constant approximation of the compressibility factor $Z(\eta,p) = k \in \mathbb{R}$ (cf., \eqref{eq:compressibilityConstant} and \cite{Pfetsch2024,SchusterDiss2021,Gugat2024,Ulke2025}), we obtain:
		\begin{equation*}
			\int \frac{p}{Z(\eta,p)} \dd p = \frac{1}{2k} p^2.
		\end{equation*}
		\item For a linear approximation of the compressibility factor $Z(\eta,p) = 1 + \alpha p$, where $\alpha$ is given by the convex combination of the gas specific constants $\alpha_{\text{H}_2}$ and $\alpha_{\text{NG}}$ (cf., \eqref{eq:compressibilityLinear} and \cite{Brodskyi2025,Schmidt2016,Gugat2018,Gugat2018b}), we obtain
        \begin{align*}
			\int \frac{p}{Z(\eta,p)} \dd p = \frac{p}{\alpha_1 \eta + \alpha_2 ( 1 - \eta)} 
			- \frac{\ln(\vert Z(\eta, p) \vert)}{\left[ \alpha_1 \eta + \alpha_2 ( 1 - \eta) \right]^2}.
		\end{align*}
        \item For a quadratic approximation of the compressibility factor \eqref{eq:compressibilityQuadratic} (cf., \textit{Papay's formula}, see \cite{Bermudez2022,Schmidt2016,WaltherHiller2017}), where $T_c(\eta_\mol)$ and $p_c(\eta_\mol)$ are critical temperature and pressure, obtained by a convex combination of the critical temperature and pressure of the components, the integral in \eqref{eq: F} reads
        \begin{align*}
			\int \frac{p}{Z(\eta_\mol, p)} \dd p 
			& = \frac{\ln(\vert Z(\eta_\mol, p) \vert)}{2 b} - \frac{a}{b \sqrt{4b - a^2}} \arctan\left(\frac{2b x + a}{\sqrt{4b - 2}}\right), \\
			\text{where} \quad a  & = \frac{1}{p_c(\eta_\mol)} \exp\left(-\frac{T}{T_c(\eta_\mol)}\right), \quad
			b  = \frac{1}{p_c(\eta_\mol)^2} \exp\left(-\frac{T}{T_c(\eta_\mol)} \right).
		\end{align*}
	\end{itemize}
\end{remark}

We obtain the pressure profile \(p(x)\) from \cref{eq: potential formulation} if the function \( F\) is invertible with respect to \(p\).
In the regularity analysis of \(F\), we make extensive use of one key physical property: gas networks operate under subsonic conditions, i.e., the gas velocity \(v\) is smaller than the sound speed \(\sigma\) in the gas, to ensure system safety and prevent structural damage \cite{Gas-Modell:DomschkeHillerLangTischendorf2017,Gugat2015}. 
For subsonic flow, the Mach number \( \mach \coloneqq v / \sigma  \) satisfies \( \mach < 1\), which leads to:
\begin{align}\label{eq: subsonic flow}
	\mach^2 = \left( \frac{v}{\sigma} \right)^2 < 1 
	\quad \Leftrightarrow \quad
	\frac{R \, T}{M(\eta)} \, \frac{q^2}{p^2} \left[ Z(\eta, p) - p \cdot \partial_p Z(\eta, p) \right] < 1,
\end{align}
where we used the pressure law~\labelcref{eq: pressure law}, the relation \( q = v \rho\), and the following expression for the sound speed:
\begin{align*}
	\left( \frac{1}{\sigma} \right)^2 
	= \partial_p \, \rho
	= \partial_p \left[ \frac{M(\eta)}{R \, T} \, \frac{p}{Z(\eta, p)}\right]
	= \frac{M(\eta)}{R \, T} \, \frac{Z(\eta, p) - p \cdot \partial_p Z(\eta, p)}{Z(\eta, p)^2}.
\end{align*}
Since we are in a subsonic regime, we define the domain of subsonic flow by multiplying \cref{eq: subsonic flow} with \(p^2\):
\begin{equation*}
	D_{\mathrm{subsonic}} \coloneqq \bigg\{ (\eta, q, p) \in [0,1] \times \R \times [0,\infty) \mathrel{\bigg|} p^2 - \frac{R \, T}{M(\eta)} q^2 \left[ Z(\eta, p) - p \, \partial_p Z(\eta, p) \right]	> 0 \bigg\}.
\end{equation*}

We proceed by proving that the function \(F\) is strictly monotonically increasing with respect to \(p\) for subsonic flows. 

\begin{lemma}\label{lemma: F}
	Let \( D_F \subseteq \{ (\eta, q, p) \in [0,1] \times \R \times [0,\infty) \mid Z(\eta, p) > 0\}  \cap D_{\mathrm{subsonic}}\) and define
	\begin{align*}
		F \colon D_F \rightarrow \R \quad  \text{with} \quad 
		F(\eta, q, p) & = \int \frac{p}{Z(\eta, p)} \dd p + \frac{R \, T \, q^2}{M(\eta)} \ln \left( \bigg\vert\frac{Z(\eta, p)}{p} \bigg\vert \right).
	\end{align*}
	Then, the function \( F\) is differentiable and strictly monotonically increasing on \( D_F \).
\end{lemma}
\begin{proof}
	The differentiability of \(F\) follows directly from its derivation since it arises as an anti-derivative when solving the differential equation~\labelcref{eq: pressure ODE}. The derivative of \( F\) with respect to \( p\) is given by
	\begin{align*}
		\partial_p F(\eta, q, p) =  \frac{p}{Z(\eta, p)} 
		+ \frac{R \, T}{M(\eta)} \, q^2 \left[ \frac{\partial_p Z(\eta, p)}{Z(\eta, p)} - \frac{1}{p} \right] 
		= \frac{p^2 - \frac{R \, T}{M(\eta)} \, q^2 \, \left[Z(\eta, p) -  p \cdot \partial_p Z(\eta, p)\right]}{p \cdot Z(\eta, p)}.
	\end{align*}	
	Since we are in a subsonic regime and due to the choice of \( D_F\), we have
	\begin{align*}
		p^2 - \frac{R \, T}{M(\eta)} q^2 \left[ \partial_p Z(\eta, p) - Z(\eta, p) \right] > 0
		\quad \text{and} \quad 
		Z(\eta, p) \ge 0.
	\end{align*}
	Thus, \( \partial_p F(\eta, q, p) > 0\) holds, which means that \( F\) is strictly monotonically increasing.
\end{proof}
Together with the Implicit Function Theorem, \cref{lemma: F} implies that \(F\) is locally bijective with respect to the pressure \(p\).

\begin{korollar}\label{lemma: inverse of F}
	The function \( F \colon D_F \rightarrow \R\) defined in \cref{lemma: F} is locally bijective with respect to the pressure \( p\). Further, the corresponding inverse \( (\eta, q, y) \mapsto g(\eta, q, y)\) is continuous and satisfies:
	\begin{equation*}
		F(\eta, q, p) = y \quad \Leftrightarrow \quad p = g(\eta, q, y).
	\end{equation*} 
\end{korollar}
\begin{proof}
	The claim follows immediately from the Implicit Function Theorem applied to the function \( f(\eta, q, p, y) = F(\eta, q, p) - y\) since \( \partial_p f(\eta, q, p, y) = \partial_p F(\eta, q, p) \) and according to \cref{lemma: F}, we have \( \partial_p F(\eta, q, p) \neq 0.\)
\end{proof}

In summary, we recover the pressure profile \( p(x)\) from \cref{eq: potential formulation} by applying the inverse \( p \mapsto g(\eta, q, p)\). Then, the pressure profile \(p(x)\) on each pipe \( e \) has the following form:

\begin{lemma}\label{lemma: p}
	Let \( q \in \R\) be the gas flow and \( \eta \in [0,1]\) the mass fraction. Then the solution to the pressure ODE~\labelcref{eq: pressure ODE} reads
	\begin{align*}
		p(x) & = g\left( \eta, q, F(\eta, q, p(0)) 
		- \frac{\lambda_\fr}{2 D} \frac{R \, T}{M(\eta)} q \vert q \vert \, x \right).
	\end{align*}
\end{lemma}

\begin{remark}
    For friction-dominated pressure loss, the nonlinear term \( \partial_x \left(q^2 / \rho \right)\) can be neglected, resulting in the semi-linear isothermal Euler equations (see, e.g., \cite{Gas-Modell:DomschkeHillerLangTischendorf2017, KochEtAl2015}. Solving the pressure loss for steady states of the semi-linear equations is analogous to the derivation above, but the function \(F\) is replaced by 
    \begin{equation*}
		F_{\mathrm{semi}}(\eta, q, p) = \int \frac{p}{Z(\eta, p)} \dd p.
	\end{equation*}
	The domain \( D_F\) in \cref{lemma: F} satisfies \(Z(\eta, p) > 0\). Thus, the function \(	F_{\mathrm{semi}}\) is also strictly monotonically increasing with respect to \(p\) and the results of \cref{lemma: F} and \cref{lemma: p} also apply for the semi-linear isothermal Euler equations. 
\end{remark}

\subsection{Gas Flow across Junctions}

Gas flow across nodes is governed by \emph{coupling conditions}, see, e.g., \cite{BandaHertyKlar2006,Gas-Modell:DomschkeHillerLangTischendorf2017,KochEtAl2015}. For gas mixtures, these consists of the following:
\begin{enumerate}[label=(\roman*)]
	\item The mass of the mixture as well as the individual constituents is conserved.
	\item The gas mixes perfectly and instantaneously, i.e., one obtains a homogeneous mixture without time delay.
	\item The pressure is continuous across a node.
\end{enumerate}
Mass conservation and pressure continuity also appear in single gas flows, while the perfect mixing condition is unique for the flow of gas mixtures. 


\medbreak
\textbf{Mass Balance.} Let $b_v \in \mathbb{R}$ be the \emph{load} of node $v \in \mathcal{V}$, i.e., the amount of gas entering or leaving the network. We have
\begin{align*}
	b_v & < 0 \quad \Leftrightarrow \quad  v \in \V_{< 0}, \text{ which means } v \text{ is a supply node,}  \\
	b_v & \ge 0 \quad \Leftrightarrow \quad v \in \V_{\ge 0}, \text{ which means } v \text{ is a demand node.}
\end{align*}
Then, the mass balance across node $v \in \mathcal{V}$ refers to the fact that the sum of all incoming flows and all outgoing flows is equal to the load. Using the incidence matrix \( A \in \R^{\vert\V\vert \times \vert \E\vert}\), cf. \cref{eq: incidence matrix}, the mass balance for the network is given by the linear system
\begin{equation*}
	\sum_{e \in \E} a(v,e) q_e = b_v \quad \text{for all } v \in \V \quad \Leftrightarrow \quad A \bm{q}=  \bm{b},
\end{equation*}
where \(\bm{b} \in \R^{\vert \V\vert}\) is the load vector with entries \(b_v\), and \(\bm{q} \in \R^{\vert \E\vert}\) the flow vector with entries \(q_e\). 

\medbreak
\textbf{Perfect Mixing and Mass Balance of the Constituents.} Perfect and instantaneous mixing means that incoming gas immediately forms a homogeneous mixture at the node before flowing into the outgoing pipes. As a result, the composition of the outgoing gas is identical across all pipes, which automatically ensures the mass balance of the individual constituents. Since the flow direction and the edge orientation can differ, it is unclear a priori which pipes carry incoming and which carry outgoing flow. However, a pipe \( e\) supplies gas to the node if \(a(v,e) q_e \ge  0\), and it transports gas away from the node if \(a(v,e) q_e < 0\). For each supply node \(v \in \mathcal{V}_{<0}\), let \( \zeta_v \in [0,1]\) be mass fraction of the first constituent of the supply node. Then, the nodal mass fraction \( \eta_v\) becomes
\begin{equation*}
	\eta_v = \frac{\sum_{e \in \E(v)} \eta_e \left( a(v,e) q_e\right)^+ + \zeta_v b_v^-}{\sum_{e \in \E(v)} \left(a(v,e)q_e\right)^+ + b_v^-},
\end{equation*}
where for a scalar \( \alpha \in\R\), we define \( \alpha^+ = \max\{\alpha, 0\}\) and \( \alpha^- = \max\{-\alpha, 0\}\). 
Additionally, we have the following relation between the nodal mass fraction \( \eta_v\) and the mass fraction \( \eta_e\) along a pipe:
\begin{equation}\label{eq: compostion on edge}
	\eta_e = \begin{cases}
		\eta_{f(e)}, & \text{if } q_e \ge 0, \\
		\eta_{h(e)}, & \text{if } q_e < 0,
	\end{cases}
\end{equation}
which encodes that the mass fraction along a pipe is always given by the upstream nodal composition. \\

\textbf{Pressure Continuity.} Pressure continuity ensures that all pipes connected to a node \( v\) have the same pressure at the shared endpoint. To enforce this, we introduce the nodal pressure \( p_v\), leading to the condition:
\begin{equation*}
	p_e(x_e) = p_v \quad \text{for all} \quad e \in \E(v) \quad \text{where} \quad 
	x_e = \begin{cases}
		0, & e \in \E_-(v), \\
		L_e, & e \in \E_+(v).
	\end{cases}
\end{equation*}
Together with \cref{eq: potential formulation}, which describes the pressure loss between along a pipe \(e\), we obtain the pressure loss between two connected, adjacent nodes:
\begin{equation*}
	F(\eta_e, q_e, p_{h(e)}) - F(\eta_e, q_e, p_{f(e)}) = -\frac{\lambda_\fr}{2 D} \frac{R \, T}{M(\eta_e)} L_e q_e \vert q_e \vert.
\end{equation*} 

We recall that the specific expression for \( F\) depends on the compressibility factor \( Z\), cf. \cref{eq: F}.

\subsection{Operation of Compressors}
Compressors play a key role in the operation of gas networks as they can increase the pressure locally.
Since pipeline networks need to operate within a certain pressure range to ensure safe operation, and since the transport of gas alone causes a drop in pressure due to friction, compressors are needed to counteract the pressure loss and to maintain a safe pressure level.

Each compressor \( e \in \E_\C\) is equipped with a so-called \emph{boosting} or \emph{compressor ratio} \( \gamma_e \ge 1\) and we assume that the compressor ratio is independent of the gas composition, cf. \cite{GugatSchuster2018,Kazi2024}.
The operation of a compressor is governed by the pressure boost equation:
\begin{equation}
	\label{eq: compressor control}
	p_{h(e)} = \gamma_e \, p_{f(e)} \quad \text{where} \quad e = (f(e), h(e)).
\end{equation}
Compressor stations are designed to increase the pressure only in one direction, gas flow in the opposite direction is not possible, cf. \cite{Gas-Modell:DomschkeHillerLangTischendorf2017,KochEtAl2015}. Thus, we make the following assumption:

\begin{assumption}\label{assumption2}
	For every edge $e \in \E_\C$, the gas flows from node \(f(e)\) to node \( h(e)\), i.e., we have \(q_e \geq 0\) for all $e \in \E_\C$.
\end{assumption}

Thus, the composition of the gas in the compressor is given by the upstream composition \( \eta_e = \eta_{f(e)}\).
Furthermore, we prohibit compressors from forming pathological structures such as cycles or running parallel to other edges to prevent circular gas flow in a cycle.

\subsection{Full Gas Flow Model}
Summarizing, the steady state gas flow model for mixtures includes the pressure loss along pipes together with the pressure continuity across nodes, the compressor property, the mass balance and the mixing condition. This is given by
\begin{subequations} \label{eq: full model}
	\begin{align}
		F(\eta_e, q_e, p_{h(e)}) - F(\eta_e, q_e, p_{f(e)}) 
		= -\frac{\lambda_\fr}{2 D} \frac{R \, T}{M(\eta_e)} L_e q_e \vert q_e \vert 
		& \quad \text{for all } e \in \E_\pipe, 
		\vphantom{\frac{\sum_{e \in \E(v)} \eta_e \left( a(v,e) q_e\right)^+ + \zeta_v b_v^-}{\sum_{e \in \E(v)} \left(a(v,e)q_e\right)^+ + b_v^-}}  
		\label{eq: model pressure loss}\\
		p_{h(e)}  = \gamma_e p_{f(e)} 
		& \quad \text{for all } e \in \E_\C,  
		\vphantom{\frac{\sum_{e \in \E(v)} \eta_e \left( a(v,e) q_e\right)^+ + \zeta_v b_v^-}{\sum_{e \in \E(v)} \left(a(v,e)q_e\right)^+ + b_v^-}}
		\label{eq: model compressor}\\
		\sum_{e \in \E_+(v)} q_e - \sum_{e \in \E_-(v)} q_e = b_v
		& \quad \text{for all } v \in \V,
		\vphantom{\frac{\sum_{e \in \E(v)} \eta_e \left( a(v,e) q_e\right)^+ + \zeta_v b_v^-}{\sum_{e \in \E(v)} \left(a(v,e)q_e\right)^+ + b_v^-}} 
		\label{eq: model mass balance}\\
		\eta_v  = \frac{\sum_{e \in \E(v)} \eta_e \left( a(v,e) q_e\right)^+ + \zeta_v b_v^-}
		{\sum_{e \in \E(v)} \left(a(v,e)q_e\right)^+ + b_v^-} 
		& \quad \text{for all } v \in \V.
		\label{eq: model mixing}
	\end{align}
\end{subequations}
In comparison to the setting in \cite{Ulke2025}, the gas flow model~\labelcref{eq: full model} also accommodates compressor stations and non-constant compressibility factors. 
The major difference is \cref{eq: model pressure loss} as the function \(F\) is no longer explicitly available for arbitrary compressiblity factors. 
For ideal gas, where \(Z=1\) holds, we have \(F(\eta, q,p) = \frac{1}{2} p^2\), which recovers \cite[Equation 2.15]{Ulke2025}.

\section{Existence of Steady-State Solutions}
We show the existence of solutions to the model~\labelcref{eq: full model} for networks with up to one cycle by extending the approach presented in \cite{Gugat2018} for the gas flow of one constituent using the quasi-linear model and assuming real gas, and in \cite{Ulke2025} for gas mixtures made up of two constituents using the semi-linear model and assuming ideal gas.
Compared to the setting in \cite{Gugat2018,Ulke2025}, the novelty lies in generalizing the existence result to arbitrary compressibility factors \( Z\) instead of assuming a constant or linear compressibility factor.
This generalization comes with a major difficulty: for non-constant compressibility factors, the pressure change along a path connecting two nodes lacks a closed-form expression in general. 
We will indicate in detail where and how this lack of an explicit representation introduces additional challenges in the proof compared to ideal gas.

\subsection{Tree-Shaped Networks}
The proof for tree-shaped networks is straightforward since the model equations decouple and we can solve the model equations subsequently: first for the flow \(q_e\), then for the mass fraction \( \eta_v\), and finally for the pressure \( p_v\).

\begin{theorem}[Existence of Solution for Trees]\label{thm: existence tree}
	Let \( \G = (\V, \E) \) be a tree-shaped network and let \(v^{\ast} \in \V\) be an arbitrary yet fixed node. Further, provide 
	\begin{enumerate}[label=(\roman*)]
		\item the load \( b_v\) for all \( v \in \V\) satisfying \( \sum_{v \in\V} b_v = 0\),
		\item the supply composition \( \zeta_v \in [0,1]\) for all \( v\in \V\) with \( b_v \le 0\), 
		\item the compressor ratio \( \gamma_e \ge 1 \) for all \( e \in \E_\C\), and
		\item the nodal pressure \( p_{v}\) for \( v = v^{\ast}\).
	\end{enumerate}
	Further, assume that the compressibility factor \(Z\) satisfies \cref{assumption} and~\labelcref{assumption 2} and that the gas flow is subsonic.
	Then the steady-state mixture model admits a unique solution. 
\end{theorem}
\begin{proof}
	The proof for the existence of unique solutions for tree-shaped networks is analogous to the proof of \cite[Theorem 3.5]{Ulke2025}. However, we must additionally exploit that the function \( p \mapsto F(\eta, q, p)\) is locally bijective (\cref{lemma: F}) for fixed \( \eta \in [0,1]\) when solving the pressure loss equation~\labelcref{eq: model pressure loss}.
\end{proof}

\subsection{Networks With One Cycle}

The presence of a cycle in a network makes proving the existence of steady states more intricate. In the following, we state the main theorem and provide the auxiliary results necessary to prove it.

\begin{theorem}[Existence of Solution for Networks with one Cycle]\label{thm: existence cycle}
	Let \( \G = (\V, \E) \) be a network with one cycle and let \(v^{\ast} \in \V\) be an arbitrary yet fixed node. Further, provide 
	\begin{enumerate}[label=(\roman*)]
		\item the load \( b_v\) for all \( v \in \V\) satisfying \( \sum_{v \in\V} b_v = 0\),
		\item the supply composition \( \zeta_v \in [0,1]\) for all \( v\in \V\) with \( b_v \le 0\), 
		\item the compressor ratios \( \gamma_e \ge 1\) for all \( e \in \E_\C\), and
		\item the nodal pressure \( p_{v}\) for \( v = v^{\ast}\).
	\end{enumerate}
	Furthermore, assume that the compressibility factor \(Z\) satisfies \cref{assumption} and~\labelcref{assumption 2} and that the gas flow is subsonic.
	Then the steady-state mixture model admits at least one solution. 
\end{theorem}

	%
	%

The idea to prove \cref{thm: existence cycle} is to cut an edge \(e^c \) of the cycle to obtain a tree-shaped network that has the same structure as the original, but without the cut edge \(e^c \) and with two additional nodes, \( v_\cl\) and \( v_\crr\) respectively, and two additional edges \( e_\cl = (f(e^c), v_\cl)\) and \( e_\crr = (v_\crr, h(e^c))\) generated by the cut.
The resulting network is the so-called \emph{cut network} \( \G^c \) of \(\G\) with respect to \( e^c\), cf. \cite{Gugat2018, WintergerstDiss}, and is defined by
\begin{align*}
	\G^c = (\V^c, \E^c) \quad \text{with} \quad 
        \V^c = \V \cup \{ v_\cl, v_\crr\},
	   \quad
	   \E^c = \left(\E \setminus \{ e^c\} \right) \cup \{ e_\cl, e_\crr\}.
\end{align*}
%
%

The goal is to find solutions on the cut graph that are also solutions on the original graph.
On the cut graph, such solutions have constant flow, mass fraction, and pressure through the cut.
To extract these solutions, we must determine suitable boundary data -- consisting of the load and supply composition -- for the new nodes generated by the cut. Throughout this paper, we indicate variables that are associated with the cut graph using the superscript \(c\).

As the relation \( b_{v_\cl}^c= - b_{v_\crr}^c\) enforces a constant flow through the cut automatically, we introduce the parameter \( \lambda\) and set \( b_{v_\cl}^c = \lambda\) and \( b_{v_\crr}^c = - \lambda\). Additionally, we introduce the parameter \( \mu\) to model the supply composition either of node \(v_{\cl}\) or node \( v_{\crr}\) depending on the sign of \( \lambda\).
Then, the parameter-dependent boundary data on the cut network reads:
\begin{align}\label{eq: cut network boundary data}
	p_{v^{\ast}}^c = p^{\ast},  \quad
	b^c_v = \begin{cases}
		b_v, & v \in \V, \\
		-\lambda, & v = v_\crr, \\
		\lambda, & v = v_\cl,
	\end{cases}  \quad \text{and} \quad 
	\zeta_v^c = \begin{cases}
		\zeta_v, & v \in \V_{<0}, \\
		\mu, & v = v_\crr \text{ and } \lambda \ge 0, \\
		\mu, & v = v_\cl \text{ and } \lambda < 0.
	\end{cases}
\end{align}

The remaining two conditions -- constant mass fraction and constant pressure through the cut -- form a non-linear system that need to be solved for the parameters \( \lambda\) and \( \mu\):
\begin{subequations}\label{eq: Hp and Heta}
	\begin{align}
		\H_{\eta}(\lambda, \mu) & \coloneqq \eta_{v_\crr}^c(\lambda, \mu)  - \eta_{v_\cl}^c(\lambda, \mu) = 0 \\ 
		\H_p(\lambda, \mu) & \coloneqq p_{v_\crr}^c(\lambda, \mu)  - p_{v_\cl}^c(\lambda, \mu) = 0. 
	\end{align}
\end{subequations}
The equation \( \H_{\eta} = 0\) ensures constant mass fraction through the cut, while the equation \( \H_p = 0\) implies that the pressure is continuous though cut. 
Thus, showing that the non-linear system \labelcref{eq: Hp and Heta} admits at least one solution \((\lambda^\ast, \mu^\ast) \in \R \times [0,1]\) is equivalent to showing that there exists at least one solution to the gas flow model \labelcref{eq: full model} for networks with one cycle, cf. \cite[Lemma 3.8]{Ulke2025}.
The proof to show the solvability of the non-linear system \labelcref{eq: Hp and Heta} follows the same fundamental argument developed in \cite{Ulke2025} for the setting of ideal gas together with the semi-linear pressure equation. The core idea is as follows:
\begin{enumerate}[label=(\roman*)]
	\item Fix a cut edge \( e^c \in \E  \). Since the flow on the original network must be acyclic, we can restrict the admissible values for the solution \( \lambda^\ast \) to an interval \( I_\sol \).
	
	\item For each fixed \( \lambda \in I_\sol \), the function \( \H_\eta(\lambda,\mu) \) admits a unique root \( \mu_\eta(\lambda) \).
	
	\item The function \( g(\lambda) = \H_p(\lambda, \mu_\eta(\lambda))\) satisfies the intermediate value theorem on the interval \( I_\sol\), i.e., 
	\begin{enumerate}[label=(\alph*)]
		\item the function \(g\) is continuous on \(I_\sol\), and
		\item it attains opposite signs at the boundary values of \(I_\sol\).
	\end{enumerate}
\end{enumerate}

Although the overall structure of the proof remains the same for our setting, the non-constant compressibility factor \(Z\) introduces significant technical difficulties. 
\cref{fig: overview proof} illustrates the structure of the proof and highlights the subtle yet significant argument that fails for arbitrary compressibility factors.
The main difficulty arises in proving continuity. The idea is to show that the root curve \(\mu_\eta(\lambda)\) and the function \(\H_p(\lambda, \mu)\) are continuous in all arguments, since then \(g\) is continuous as composition of continuous functions.
The composition on the cut graph is independent of the pressure, and hence of the compressibility factor \(Z\), as the cut graph is tree-shaped. Therefore, the derivation of \(\mu_\eta\) is analogous to the case of ideal gas, cf. \cite[Lemma 3.20]{Ulke2025}. For all \(\lambda \in \R\) such that gas does not flow from node \(v_\cl\) to node \(v_\crr\), we obtain the following expression:
\begin{equation*}
	\mu_\eta(\lambda) = \begin{cases}
		\eta_{v_\crr}^c(\lambda), & \lambda < 0, \\
		\eta_{v_\cl}^c(\lambda), & \lambda \ge 0.
	\end{cases}
\end{equation*}

%

Since the root curve \(\mu_\eta\) is given by the composition at either node \(v_\cl\) or node \(v_\crr\), its continuity follows from the continuity of the composition.
The continuity of the composition does not depend on the compressibility factor \(Z\), whereas the continuity of \(\H_p\) does, cf. \cite[Lemma 3.18, Lemma 3.21]{Ulke2025}. 
For ideal gas, where \( Z = 1\), there exists an explicit expression of \( \H_p\) in terms of the flow and the composition. This explicit dependence allows us to derive the continuity of \(\H_p\) from the continuity of the flow and the composition, as provided by \cite[Corollary 3.15, Lemma 3.16, Lemma 3.17]{Ulke2025}.
However, there is one issue: the composition at the nodes generated by the cut is only right-sided continuous. 
Fortunately, the explicit dependence makes it possible to analyze the effect of the one-sided continuity at \( \lambda = 0\) on the continuity of \(\H_p\), revealing that the discontinuity in fact vanishes.

\begin{figure}[ht]
	\centering
	\scalebox{0.72}{
	\begin{tikzpicture}[
		> = {Latex[scale=1.0]}, 
		]
		
		\tikzstyle{dataset} = [rectangle, draw, text width = 2in, rounded corners = 5pt, thick]
		
		\pgfmathsetmacro{\hx}{7}
		\pgfmathsetmacro{\hxx}{4.25}
		\pgfmathsetmacro{\hxxx}{6}
		\pgfmathsetmacro{\hy}{1.75}
		\pgfmathsetmacro{\hyy}{1.75}
		
		\pgfmathsetmacro{\d}{0.3}
		\pgfmathsetmacro{\dx}{1.5}
		\pgfmathsetmacro{\dxx}{0.5}
		\pgfmathsetmacro{\dy}{1.5}
		\pgfmathsetmacro{\dyy}{\dy/2}
		
		\node[above] at (\hx/2, {3*\hy}) {\textbf{Idea of Proof}};
	
		\draw[rounded corners=5pt, thick]
		(0,0) rectangle (\hx, 3*\hy)
		node[rectangle split, rectangle split parts=2, midway, text width=\hx cm, align=left]
		{
		\begin{enumerate}[label=(\roman*)]
			\itemsep = 15pt
			\item Restrict the admissible values for the solution \( \lambda^\ast \) to \( I_\sol \).
			\item For fixed \( \lambda \in I_\sol \), the function \( H_\eta(\lambda,\mu) \) has a unique root \( \mu_\eta(\lambda) \).
			\item Apply the intermediate value theorem to \( g(\lambda) \) on \( I_\sol \).
		\end{enumerate}
		};

		\draw[rounded corners=5pt, thick]
		(\hx + \dx, 0.5*\hy + \dyy) rectangle (\hxx + \hx + \dx, 1.5*\hy + \dyy)
		node[rectangle split, rectangle split parts=2, midway]
		{ \(g\) attains  opposite signs
			\nodepart{second}
			at the boundary of \(I_\sol\).};
		
		\draw[rounded corners=5pt, thick]
		(\hx + \dx, -0.5*\hy - \dyy ) rectangle (\hxx + \hx + \dx, 0.5*\hy - \dyy)
		node[rectangle split, rectangle split parts=2, midway]
		{The function \(g \) 
			\nodepart{second}
			is continuous on \(I_\sol\).};

		\draw[rounded corners=5pt, ultra thick, color=red!80!black]
		(\hxx + \hx + 2*\dx, 0.5*\hy) rectangle (\hxxx + \hxx + \hx + 2*\dx,  1.5*\hy)
		node[rectangle split, rectangle split parts=2, midway]
		{\(\bm{\H_p}\) \textbf{has an explicit expression} 
			\nodepart{second}
			\textbf{in terms of} \(\bm{q_e^c}\) \textbf{and} \(\bm{\eta_v^c}\)};

		\draw[rounded corners=5pt, thick]
		(\hxx + \hx + 2*\dx, -0.5*\hy - \dyy) rectangle (\hxxx + \hxx + \hx + 2*\dx,  0.5*\hy - \dyy)
		node[rectangle split, rectangle split parts=2, midway]
		{The flow \(q_e^c\) \nodepart{second} is continuous.};
		
		\draw[rounded corners=5pt, thick]
		(\hxx + \hx + 2*\dx, -1.5*\hy - 2*\dyy) rectangle (\hxxx + \hxx + \hx + 2*\dx,  -0.5*\hy - 2*\dyy)
		node[rectangle split, rectangle split parts=2, midway]
		{\(\eta_v^c \) is continuous in \(\lambda \neq 0\) and
			\nodepart{second}
			right-sided continuous in \(\lambda = 0\).};
		
		
		\draw[<-, thick] (\hx, 0.66*\hy) -- (\hx + \dx,  \hy + \dyy);
		\draw[<-, thick] (\hx, 0.33*\hy) -- (\hx + \dx, - \dyy);
		
		\draw[<-, ultra thick, red!80!black] (\hxx + \hx + \dx, -0.5*\hy - \dyy + 0.75*\hy) -- (\hxx + \hx + 2*\dx,  \hy)
		node[midway, above, sloped] {\textbf{IDEAL}}
		node[midway, below, sloped] {\textbf{GAS}};
		\draw[<-, thick] (\hxx + \hx + \dx, - \dyy) -- (\hxx + \hx + 2*\dx,  - \dyy);
		\draw[<-, thick] (\hxx + \hx + \dx, -0.5*\hy - \dyy + 0.25*\hy) -- (\hxx + \hx + 2*\dx, -\hy - 2*\dyy);
		
	\end{tikzpicture}
}
	\caption{Key idea to prove existence of steady-states, and step at which the argumentation of \cite[Lemma 3.13]{Ulke2025} for the ideal gas case fails for an  arbitrary compressibility factor \(Z = Z(\eta, p)\) (highlighted in red).}
	\label{fig: overview proof}
\end{figure}
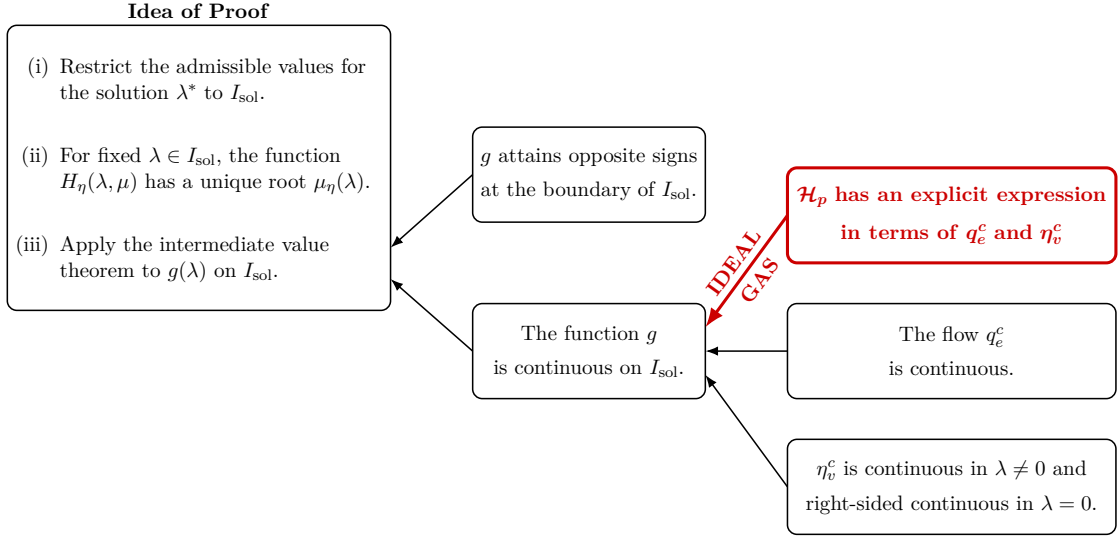

For an arbitrary compressibility factor \(Z\), the explicit dependence of \(\H_p\) on the flow and composition is unknown. Thus, it remains unclear how the discontinuity enters \(\H_p\) and it is not possible to verify that \(\H_p\) is also continuous at \( \lambda = 0\) immediately.
To overcome this difficulty, we proceed with the following steps:

\begin{itemize}
	\item We prove the continuity of \(\H_p\) for \(\lambda \neq 0\) and right-sided continuity at \(\lambda = 0\) using an induction argument.
	\item We show that a cut edge exists such that \( I_\sol \subset [0,\infty) \). This choice allows us to invoke the intermediate value theorem for \(g\) on \(I_\sol\) immediately.
\end{itemize}

\subsection{Continuity of Flow, Mass Fraction, and Pressure}

Besides ensuring continuous pressure across the cut edge, the function \(\H_p\) also describes the pressure change along the path connecting the nodes \(v_\crr\) and \(v_\cl\).
We aim to establish the continuity of \(\H_p\) by an induction based argument utilizing \cref{eq: model pressure loss}. Therefore, we recall the continuity properties of the flow and composition.
As the cut network is tree-shaped, the continuity result \cite[Corollary 3.11]{Ulke2025} for ideal gas applies directly to our setting and provides the continuity of the flow.
For the continuity of the mass fraction, we have:


\begin{lemma}[{Continuity of the Mass Fraction}]\label{lemma: composition continuous}
	Let \( \G = (\V, \E)\) be a connected, directed graph with a cycle. Further, let \( e\) be an edge of the cycle and \( \G^c\) be the corresponding cut graph with boundary data as defined in \cref{eq: cut network boundary data}. Then the mass fraction \( \eta_v^c = \eta_v^c(\lambda,\mu)\) is 
	\begin{enumerate}[label=(\roman*)]
		\item continuous in \(\mu\) for all \( v \in \V^c\),
		\item continuous in \(\lambda \neq 0 \) for all \( v \in \V^c \),
		\item continuous in \( \lambda  = 0\) for all \( v \in \V^c \setminus \{ v_\cl, v_\crr\}\),
		\item right-sided continuous in \( \lambda = 0\) for \( v \in \{ v_\cl, v_\crr\}\).
	\end{enumerate}
\end{lemma}
\begin{proof}
	The continuity with respect to \( \mu\) follows directly from \cite[Lemma 3.13]{Ulke2025}, while the continuity with respect to \( \lambda\) in \(\lambda \neq 0\) for all \( v\in \V^c\) and in \( \lambda = 0\) for all \( v \notin \{v_\cl, v_\crr\}\) follows from \cite[Lemma 3.12]{Ulke2025}. 
	Furthermore, we can additionally exploit that the boundary data \labelcref{eq: cut network boundary data} is right-sided continuous in \( \lambda = 0\), which yields together with \cite[Lemma 3.12]{Ulke2025} the right-sided continuity in \( \lambda = 0\) for \( v \in \{v_\cl, v_\crr\} \).
\end{proof}

For ideal gas, the continuity of the flow and the mass fraction immediately implies the continuity of \(\H_p\), since in this case the pressure change along a path admits a closed-form representation that depends only on the flow and the mass fraction, cf. \cite{Ulke2025}.
For real gases, the structure of the pressure law makes it in general impossible to express the pressure change along a path explicitly in terms of the flow and the composition. 
Nevertheless, we can still establish the continuity of \( \H_p\) with respect to \( \lambda \) and \( \mu\) with an induction based argument. 

\begin{lemma}[{Continuity of the Nodal Pressure}]\label{lemma: pressure continuous}
	Let \( \G = ( \V, \E) \) be a directed, connected graph, \( e^c \) be an edge of the cycle, and \( \G^c\) the cut  graph corresponding to the cut edge \(e^c\). Suppose boundary data according to \cref{eq: cut network boundary data} is given. Then, for all \( v \in \V^c\), the pressure \( p_v^c = p_v^c(\lambda, \mu)\) is 
	\begin{enumerate}[label=(\roman*)]
		\item continuous in \( \mu \),  
		\item continuous in \( \lambda \neq 0 \) and right-sided continuous in \( \lambda = 0.\)
	\end{enumerate}
\end{lemma}
\begin{proof}
	We fix an arbitrary node \( v \in \V^c \setminus  \{v^\ast\}\). As the cut graph is tree-shaped, there exists only one path connecting \(v\) to \(v^\ast\).
	We set the first node of the path to \( v_1 = v^\ast\) and traverse the path, successively labeling each node until we reach the end of the path, i.e., the node \( v_{n_v} = v\).
	Then, the nodes of the path are given by \(v_i\) for \( i=1,\ldots, n_v\). Further, two incident nodes \( v_i\) and \( v_{i+1}\) are connected by the edge \( e_i\); see \cref{fig: numbering path} for an illustration.
	To prove the continuity of \(p_v^c\), we proceed with an induction over the position \(i\) of nodes along the path. 
	
	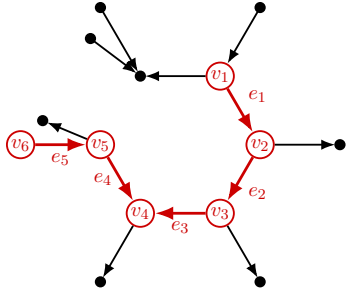
\begin{figure}[ht]
		\centering
		\newcommand{\radius}{1.5}   
\scalebox{0.7}{
	\begin{tikzpicture}[
		> = {Latex[scale=1.0]}, 
		line width=1pt,
		every node/.style={}
		]
		
		\tikzstyle{cycle}=[
		red!80!black,
		draw=red!80!black,
		circle,
		inner sep=1pt,
		minimum size=0pt
		]
		
		\tikzstyle{filled}=[
		fill=black, 
		draw=black,
		circle,
		inner sep=1pt,
		outer sep=0pt,
		minimum size=5pt
		]
		
		\tikzstyle{grey}=[
		fill=black!30, 
		draw=black!30,
		circle,
		inner sep=1pt,
		outer sep=0pt,
		minimum size=5pt
		]

		
		\node[filled] (B1) at ({\radius*cos(120)}, {\radius*sin(120)}) {};
		\node[cycle] (B2) at ({\radius*cos(60)}, {\radius*sin(60)}) {$v_1$};
		\node[cycle] (B3) at ({\radius*cos(0)}, {\radius*sin(0)}) {$v_{2}$};
		\node[cycle] (B4) at ({\radius*cos(300)}, {\radius*sin(300)}) {$v_{3}$};
		\node[cycle] (B5) at ({\radius*cos(240)}, {\radius*sin(240)}) {$v_{4}$};
		\node[cycle] (B6) at ({\radius*cos(180)}, {\radius*sin(180)}) {$v_5$};
		
		\foreach \i/\name in {0/C3, 60/C2, 120/C1, 240/C5, 300/C4} {
			\node[filled] (\name) at ({2*\radius*cos(\i)}, {2*\radius*sin(\i)}) {};
		}
		
		\node[cycle] (C6) at ({2*\radius*cos(180)}, {2*\radius*sin(180)}) {$v_6$};
		
		\draw[->] (B2) -- (B1) node[midway, above] {};
		\draw[->, red!80!black, line width = 1.5pt] (B2) -- (B3) node[midway, above right] {$e_1$};
		\draw[->, red!80!black, line width = 1.5pt] (B3) -- (B4) node[midway, below right] {$e_2$};
		\draw[->, red!80!black, line width = 1.5pt] (B4) -- (B5) node[midway, below] {$e_3$};
		\draw[->, red!80!black, line width = 1.5pt] (B6) -- (B5) node[midway, left] {$e_4$};
		
		\draw[->] (C1) -- (B1) {};
		\draw[->] (C2) -- (B2) {};
		\draw[->] (B3) -- (C3) {};
		\draw[->] (B4) -- (C4) {};
		\draw[->] (B5) -- (C5) {};
		\draw[->, red!80!black, line width = 1.5pt] (C6) -- (B6) node[midway, below] {$e_5$};
		
		\node[filled] (vr) at ({1.75*\radius*cos(130)}, {1.75*\radius*sin(130)}) {};
		\node[filled] (vl) at ({1.75*\radius*cos(170)}, {1.75*\radius*sin(170)}) {};
		
		\draw[->] (vr) -- (B1) node[midway, below left] {};
		\draw[->] (B6) -- (vl) node[midway, above right] {};
		
	\end{tikzpicture}
}
		\caption{A cut network with the path connecting the node \(v_1 = v^\ast\) to the node \(v_6 = v\). The path and the numbering are highlighted in red.}
		\label{fig: numbering path}
	\end{figure}
	
	\emph{Base case.} At position \( i = 1\), we have the node \( v^\ast\). Since the pressure at node \(v^\ast\) is given through boundary data, the pressure \( p_{v^\ast}^c\) is constant, and thus continuous in both \( \lambda\) and \( \mu\).

	\emph{Induction step.} Suppose the pressure \(p_{v_i}^c\) is continuous in \( \mu\) and in \( \lambda\neq 0\) as well as right-sided continuous in \(\lambda = 0\) for an arbitrary yet fixed \( i < n_v\). 
	The goal is to prove that -- given the continuity of \(p_{v_i}^c\) -- the pressure \(p_{v_{i+1}}^c\) is also continuous. 
	
	Therefore, we distinguish between two cases. The first case is that \( e_i = (v_i, v_{i+1})\) and the second one that \( e_i = (v_{i+1}, v_i)\).
	We  prove the first case.
	If the edge \(e_i \) is a pipe, we utilize \cref{eq: model pressure loss} to express the pressure \( p_{v_{i+1}}^c\) in terms of the pressure \( p_{v_i}^c\):
	\begin{align*}
		F(\eta_{e_i}^c, q_{e_i}^c, p_{v_{i+1}}^c) - F(\eta_{e_i}^c, q_{e_i}^c, p_{v_i}^c) 
		& = \underbrace{- \frac{\lambda_\fr}{2D} \frac{R \, T}{M(\eta_{e_i}^c)} 
			L_{e_i}^c q_{e_i}^c\vert q_{e_i}^c \vert}_{=S(\eta_{e_i}^c, q_{e_i}^c)}, 
	\end{align*}
	with \cref{lemma: inverse of F} this leads to:
	\begin{align*}
		&& F(\eta_{e_i}^c, q_{e_i}^c, p_{v_{i+1}}^c) & = F(\eta_{e_i}^c, q_{e_i}^c, p_{v_i}^c)  + S(\eta_{e_i}^c, q_{e_i}^c) \\
		\Leftrightarrow && p_{v_{i+1}}^c & = g\left(\eta_{e_i}^c, q_{e_i}^c, F(\eta_{e_i}^c, q_{e_i}^c, p_{v_i}^c)  + S(\eta_{e_i}^c, q_{e_i}^c)\right).
	\end{align*}
	Recall the dependencies \( \eta_{e_i}^c = \eta_{e_i}^c(\lambda, \mu)\), \( q_{e_i}^c = q_{e_i}^c(\lambda)\), and \( p_{v_i}^c = p_{v_i}^c(\lambda, \mu)\). 
	From \cite[Lemma 3.10]{Ulke2025} follows that \(q_e^c\) is continuous in \(\lambda \). \cref{lemma: composition continuous} together with \cref{eq: compostion on edge} implies that \(\eta_e^c\) is continuous in \(\mu\) and in \( \lambda \neq 0\) as well as right-sided continuous in \( \lambda = 0\).
	Moreover, the function \( g\) is continuous in all of its components by \cref{lemma: inverse of F}, and \( S\) is continuous continuous by definition.	
	Thus, \(p_{v_{i+1}}^c\) is continuous as a composition of continuous functions.
	If the edge \(e_i\) is a compressor, we have the relation \(  p_{v_{i+1}}^c = \gamma_{e_i} \, p_{v_{i}}^c \). 
	However, in this case the continuity of \( p_{v_{i+1}}^c\) follows immediately since \( \gamma_{e_i}\) is a constant.
	
	The case \( e_i = (v_{i+1}, v_i)\) is analogous and thus, as \(v = v_{n_v}\), we obtain that \( p_v^c\) is continuous in \(\mu\) and in \( \lambda \neq 0\) as well as right-sided continuous in \( \lambda = 0\).
\end{proof}

\begin{korollar}\label{korollar: Hp continuous}
		Let \( \G = ( \V, \E) \) be a directed, connected graph, \( e^c \) be an edge of the cycle, and \( \G^c\) the cut  graph corresponding to the cut edge \(e^c\). Suppose boundary data according to \cref{eq: cut network boundary data} is given. Then the function \(\H_p(\lambda,\mu) = p_{v_\crr}^c(\lambda, \mu) - p_{v_\cl}^c(\lambda, \mu)\) is
	\begin{enumerate}[label=(\roman*)]
		\item continuous in \( \mu \),  
		\item continuous in \( \lambda \neq 0 \) and right-sided continuous in \( \lambda = 0.\)
	\end{enumerate}
\end{korollar}

\subsection{Existence of a Suitable Cut Edge}

For the ideal gas case, we obtain the interval \(I_\sol\) utilizing \cite[Lemma 3.15]{Ulke2025}. The key observation here is that if \(\lambda\) is sufficiently small the gas flows from node \(v_\crr\) to node \( v_\cl\) in the cut graph, and if it is sufficiently large the gas flow in the opposite direction. 
On the original graph this translates into a circular gas flow. However, as the gas flow must be acyclic, we can now bound the admissible values for \(\lambda\) from below and above, which yields the interval \(I_\sol\).
The boundary values of \(I_\sol\) depend only on the flow \(q_e^c\) and loads \( b_v\) of the network as well as the network topology and the choice of the cut edge \(e^c\). Hence, \cite[Lemma 3.15]{Ulke2025} is independent from the pressure change equation and transfers to case of arbitrary compressiblity factors \(Z\).

As the interval \(I_\sol\) depends on the cut edge, we aim to find a cut edge \(e^c\) such that \( I_\sol \subset [0,\infty)\).
For ideal gases, the question of finding a \emph{suitable} cut edge does not occur because in this case the function \( \H_p\) is continuous for all \(\lambda \in \R\).
To address this question for real gases, we take a closer look on the derivation of the interval \(I_\sol\) and introduce necessary notation.

\begin{definition}[Sub-Graph of a Cycle]
	Let \( \G = (\V, \E)\) be a directed, connected graph with a cycle. Then, the cycle forms a connected sub-graph \( \G_\cycle = (\V_\cycle, \E_\cycle)\) of the graph \( \G\) where \( \V_\cycle \subseteq \V \) denotes the nodes, and \( \E_\cycle  \subseteq \E \) the edges of the cycle. Further, \( n_\cycle = \vert \V_\cycle \vert\) denotes the number of nodes in the cycle.
\end{definition}

\begin{definition}[Sub-Graph of a Cut Cycle]
	Let \( \G = (\V, \E)\) be a directed, connected graph with a cycle and let \( e^c \in \E_\cycle\) be an edge in the cycle. By cutting the edge \(e^c\), we obtain the cut graph \( \G_\cycle^c = ( \V_\cycle^c, \E_\cycle^c)\) of the cycle, where
	\begin{gather*}
		\V_\cycle^c = \V_\cycle \cup \{ v_\cl, v_\crr\}, \quad 
		\E_\cycle^c = \E_\cycle \setminus \{e^c\} \cup \{e_\cl, e_\crr\}, 
	\end{gather*}
	The cut graph \( \G_\cycle^c \) is tree-shaped and only consists of a single path connecting the two new nodes \(v_\cl\) and \(v_\crr\). It also forms a sub-graph of the cut graph \( \G^c\) of \( \G\).
\end{definition}

The flow on the edges of the cut cycle is linear in \( \lambda\) and independent of \(\mu\), cf. \cite[Lemma 3.10]{Ulke2025}, which implies that each flow \(q_e^c\) has a unique root \(\beta_e\) with respect to \(\lambda\) for all \( e \in \E_\cycle^c\).
%
Together with \cite[Lemma 3.15]{Ulke2025}, we can identify suitable boundary values \(\lambda^{\pm}\) for the interval \(I_\sol = [\lambda^-, \lambda^+]\) and summarize the finding in the following lemma.

\begin{lemma}\label{lemma: beta}
	Let \( \G=(\V,\E)\) be a directed, connected graph with one cycle and let \( e^c \in\E_\cycle\) be the cut edge. Further, denote the root of the flow \( q_e^c\) with respect to \(\lambda\) by \( \beta_e\) for \(e \in \E_\cycle^c\), cf. \cite[Lemma 3.10]{Ulke2025}, and set:
	\begin{equation*}
		\lambda^- = \min_{e\in \E_\cycle^c} \beta_e
		\quad \text{and} \quad \lambda^+  = \max_{e\in \E_\cycle^c} \beta_e,
	\end{equation*}
	Then a solution \((\lambda^\ast,\mu^\ast)\) to the non-linear system~\labelcref{eq: Hp and Heta} must satisfy \( \lambda^\ast \in \left[\lambda^-, \lambda^+\right]\).
	Furthermore, the values  \( \lambda^\pm\) satisfy the following inequalities:
	\begin{equation*}
		\H_p(\lambda^-, \mu) \le 0 \quad \text{and} \quad \H_p(\lambda^+, \mu) \ge 0
		\quad \text{for all} \quad \mu \in [0,1].
	\end{equation*}
\end{lemma}
\begin{proof}
	The proof is analogous to \cite[Lemma 3.15]{Ulke2025} for the case of an ideal gas, because the flow on the cut graph is independent of the pressure law since the cut graph is tree-shaped.
	The proof exploits the fact that pressure decreases in the flow direction, i.e., gas flow must be acyclic unless there is a compressor in a cycle, but this is prohibited by assumption. 
	For real gas, pressure also decreases in the flow direction, because with \cref{eq: model pressure loss}, we have:
	\begin{align*}
		&& F(\eta_e, q_e, p_{h(e)}) - F(\eta_e, q_e, p_{f(e)})
		& = \underbrace{- \frac{\lambda_\fr}{2 D} \frac{R \, T}{M(\eta_e)} L_e \vert q_e \vert}_{< 0} q_e \\
		\Rightarrow && \sign \left( F(\eta_e, q_e, p_{h(e)}) - F(\eta_e, q_e, p_{f(e)})\right)
		& = - \sign(q_e).
	\end{align*}
	When the gas flows from node \( f(e)\) to node \( h(e)\), we have \( q_e \ge 0\) and thus,
	\begin{equation*}
		F(\eta_e, q_e, p_{h(e)}) \le F(\eta_e, q_e, p_{f(e)}).
	\end{equation*}
	Due to the monotonicity of \(F\) with respect to the pressure \( p\), cf. \cref{lemma: F}, we get \( p_{h(e)} \le p_{f(e)}\). The case when gas flows in the opposite direction, i.e., from node \( h(e)\) to node \( f(e)\), follows analogously and results in \( p_{h(e)} \ge p_{f(e)}\) as here we have \(q_e \le 0.\)
\end{proof}

The values \( \beta_e\) depend on the choice of cut edge \(e^c\), and consequently so does the interval \(I_\sol\).
Therefore, the question whether there exists a cut edge \( e^c\) such that \( I_\sol \subset [0,\infty) \), is equivalent to finding a cut edge \(e^c\) such that \( \beta_e \ge 0\) for all \( e \in \E_\cycle^c\).

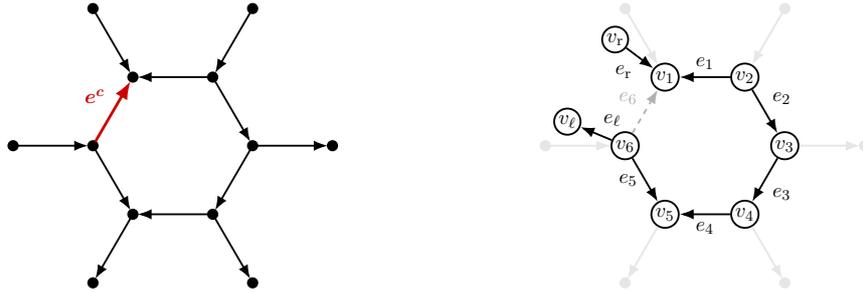
\begin{figure}[ht]
	\centering
	\begin{subfigure}[t]{0.45\textwidth}
		\centering
		\newcommand{\radius}{1.5} 

\scalebox{0.7}{
	\begin{tikzpicture}[
		> = {Latex[scale=1.0]}, 
		line width=1pt,
		every node/.style={}
		]
		
		\tikzstyle{cycle}=[
		fill=white, 
		draw=black,
		circle,
		inner sep=1pt,
		minimum size=0pt
		]
		
		\tikzstyle{filled}=[
		fill=black, 
		draw=black,
		circle,
		inner sep=1pt,
		outer sep=0pt,
		minimum size=5pt
		]
		
		\foreach \pos/\name in {0/B3, 60/B2, 120/B1, 180/B6, 240/B5, 300/B4} {
			\node[filled] (\name) at ({\radius*cos(\pos)}, {\radius*sin(\pos)}) {};
		}
		
		\foreach \i/\name in {0/C3, 60/C2, 120/C1, 180/C6, 240/C5, 300/C4} {
			\node[filled] (\name) at ({2*\radius*cos(\i)}, {2*\radius*sin(\i)}) {};
		}
		
		\draw[->] (B2) -- (B1) node[midway, above right] {};
		\draw[->] (B2) -- (B3) node[midway, right] {};
		\draw[->] (B3) -- (B4) node[midway, below right] {};
		\draw[->] (B4) -- (B5) node[midway, below left] {};
		\draw[->] (B6) -- (B5) node[midway, left] {};
		\draw[->, red!80!black, line width = 1.5pt] (B6) -- (B1) node[midway, above left] {$\bm{e^c}$};
		
		\draw[->] (C1) -- (B1) {};
		\draw[->] (C2) -- (B2) {};
		\draw[->] (B3) -- (C3) {};
		\draw[->] (B4) -- (C4) {};
		\draw[->] (B5) -- (C5) {};
		\draw[->] (C6) -- (B6) {};
		
		
	\end{tikzpicture}
}
	\end{subfigure}%
	\begin{subfigure}[t]{0.45\textwidth}
		\centering
		\newcommand{\radius}{1.5}   

\scalebox{0.7}{
\begin{tikzpicture}[
	> = {Latex[scale=1.0]}, 
	line width=1pt,
	every node/.style={}
	]
	
	\tikzstyle{cycle}=[
	fill=white, 
	draw=black,
	circle,
	inner sep=1pt,
	minimum size=0pt
	]
	
	\tikzstyle{filled}=[
	fill=black, 
	draw=black,
	circle,
	inner sep=1pt,
	outer sep=0pt,
	minimum size=5pt
	]
	
	\tikzstyle{grey}=[
	fill=black!10, 
	draw=black!10,
	circle,
	inner sep=1pt,
	outer sep=0pt,
	minimum size=5pt
	]

	\foreach \i/\pos/\name in {3/0/B3, 2/60/B2, 1/120/B1, 6/180/B6, 5/240/B5, 4/300/B4} {
		\node[cycle] (\name) at ({\radius*cos(\pos)}, {\radius*sin(\pos)}) {$v_{\i}$};
	}
		
	\foreach \i/\name in {0/C3, 60/C2, 120/C1, 180/C6, 240/C5, 300/C4} {
		\node[grey] (\name) at ({2*\radius*cos(\i)}, {2*\radius*sin(\i)}) {};
	}
		
	\draw[->] (B2) -- (B1) node[midway, above] {$e_1$};
	\draw[->] (B2) -- (B3) node[midway, above right] {$e_2$};
	\draw[->] (B3) -- (B4) node[midway, below right] {$e_3$};
	\draw[->] (B4) -- (B5) node[midway, below] {$e_4$};
	\draw[->] (B6) -- (B5) node[midway, left] {$e_5$};
	\draw[->, dashed, color=black!30] (B6) -- (B1) node[midway, above left] {$e_6$};
		
	\draw[->, color=black!10] (C1) -- (B1) {};
	\draw[->, color=black!10] (C2) -- (B2) {};
	\draw[->, color=black!10] (B3) -- (C3) {};
	\draw[->, color=black!10] (B4) -- (C4) {};
	\draw[->, color=black!10] (B5) -- (C5) {};
	\draw[->, color=black!10] (C6) -- (B6) {};
		
	\node[cycle] (vr) at ({1.75*\radius*cos(130)}, {1.75*\radius*sin(130)}) {$v_\crr$};
	\node[cycle] (vl) at ({1.75*\radius*cos(170)}, {1.75*\radius*sin(170)}) {$v_\cl$};
	
	\draw[->] (vr) -- (B1) node[midway, below left] {$e_\crr$};
	\draw[->] (B6) -- (vl) node[midway, above right] {$e_\cl$};
				
\end{tikzpicture}
}
	\end{subfigure}
	\caption{A gas network with cut edge \( e^c\) and \(n_\cycle=6\) (left), and the corresponding cut graph with nodes and edges of the former cycle being numbered (right).}
	\label{fig: numbering}
\end{figure}

In the following, it is crucial to understand the derivation and computation of the values~\( \beta_e\). 
Therefore, we fix a cut edge \( e^c\) and number the nodes and edges of the cycle. For the nodes, we set \( v_1 = h(e^c)\) and traverse through the cycle, successively labeling each subsequent node until we reach the end of the cycle, where we set \( v_{n_\cycle} = f(e^c)\); see \cref{fig: numbering} for an illustration.
%
For the edges, we follow a similar pattern: each edge \(e_i\) connects the nodes \( v_i\) and \( v_{i+1}\) for \( i =1, \ldots, n_\cycle-1\) with the exception of the edge \( e_{n_\cycle} = e^c\), which connects the nodes \(v_{n_\cycle}\) and \( v_1\) closing the cycle.


With this notation at hand, we can express the flow \(q_e^c\) on an edge of the former cycle as a linear function with respect to \(\lambda\), which introduces the values \( \beta_e\):
\begin{lemma}[{see \cite[Lemma 3.10]{Ulke2025} and \cite[Theorem 1]{Gugat2018}}]\label{lemma: flow linear}
	Let \( \G= (\E,\V)\) be a network with one cycle and let \( e^c \in \E_\cycle\) be the cut edge. 
	Then, the flow on the edges \( e \in \E_\cycle^c\) of the former cycle is linear in \( \lambda\).
	Adapting the numbering of nodes and edges illustrated in \cref{fig: numbering}, the flow \(q_{e_i}^c\) on these edges for \( i \in\{1,\ldots,n_\cycle-1\}\) reads:  
	\begin{alignat}{1}\label{eq: beta vr to vl}
		a^c(v_i,e_i) q_{e_i}^c(\lambda) = \beta_{e_i} - \lambda 
		\quad \text{where} \quad \beta_{e_i} & = \sum_{j=1}^i b_{v_j}^\P, \\
		b_v^\P & = b_v - \sum_{e \in \E(v) \setminus \E_\cycle^c} a^c(v,e) q_e^c \quad \text{for} \quad v \in \V_\cycle, \nonumber
	\end{alignat}	
	where \(a^c(v,e)\) refers to the entries of the incident matrix \(A^c\) of the cut graph \( \G^c\).
	Alternatively, the flow \( q_{e_i}^c\) can be written as
	\begin{alignat}{1}\label{eq: beta vl to vr}
		a^c(v_{i+1},e_i) q_{e_i}^c(\lambda) = \tilde{\beta}_{e_i} + \lambda 
		\quad \text{where} \quad \tilde{\beta}_{e_i} & = \sum_{j=1}^{n_\cycle -i} b_{v_{n_\cycle-j+1}}^\P.
	\end{alignat}	
	Further, the new edges \(e_\crr\) and \( e_\cl\) generated by the cut satisfy:
	\begin{equation*}
		q^c_{e_\crr}(\lambda) = \lambda \quad \text{and} \quad 
		q^c_{e_\cl}(\lambda) = \lambda 
		\quad \Rightarrow \quad
		\beta_{e_\crr} = \beta_{e_\cl} =0 \quad \text{and} \quad 
		\tilde{\beta}_{e_\crr} = \tilde{\beta}_{e_\cl} = 0.
	\end{equation*}
\end{lemma}

\begin{remark}
	In \cref{eq: beta vr to vl}, we traverse the nodes of the former cycle from \(v_\crr = v_0\) to \( v_\cl = v_{n_\cycle +1}\), whereas in \cref{eq: beta vl to vr}, we traverse the nodes in the reversed direction.
	The values \( \beta_{e_i}\) are the partial sums of the modified loads \( b_v^\P\) starting at  \( v_1\), proceeding forward along the path connecting the nodes \(v_\crr\) and \( v_\cl\).
	In contrast, the values \( \tilde{\beta}_{e_i}\) are the reversed partial sums, traversing the nodes in the opposite direction, see \cref{fig: computation beta} for an illustration.
\end{remark}

In summary, \cref{lemma: flow linear} allows us to express \(\lambda^\pm\) in terms of the values \(\beta_{e_i}\) using the notation illustrated in \cref{fig: numbering}, i.e.,
%
\begin{align*}
	\lambda^- = \min M_\beta, 
	   \quad 
	   \lambda^+ = \max M_\beta,
       \quad \text{where} \quad 
       M_\beta = \{ \beta_{e_i} \mid 1 \le i \le n_\cycle-1 \} \cup \{0\}.
\end{align*}
It also shows that a cut edge such that \(0 \notin I_\sol\) cannot exist. As the values \(\beta_{e_\crr}\) and  \(\beta_{e_\cl}\) are independent of the cut edge, we always have  \( 0 \in I_\sol\). Thus, finding a cut edge such that \( I_\sol \subset [0,\infty)\) implies that \(0\) must lie on the left boundary of \(I_\sol\).

\begin{figure}[ht]
	\centering
	\begin{subfigure}[t]{0.45\textwidth}
		\centering
		\newcommand{\radius}{1.5} 
\newcommand{\Rscale}{0.5} 

\scalebox{0.7}{
	\begin{tikzpicture}[
		> = {Latex[scale=1.0]}, 
		line width=1pt,
		every node/.style={}
		]
		
		\tikzstyle{cycle}=[
		fill=white, 
		draw=black,
		circle,
		inner sep=1pt,
		minimum size=0pt
		]
		
		\tikzstyle{filled}=[
		fill=black, 
		draw=black,
		circle,
		inner sep=1pt,
		outer sep=0pt,
		minimum size=5pt
		]
		
		\tikzstyle{grey}=[
		fill=black!10, 
		draw=black!10,
		circle,
		inner sep=1pt,
		outer sep=0pt,
		minimum size=5pt
		]

		\foreach \i/\pos/\name in {3/0/B3, 2/60/B2, 1/120/B1, 6/180/B6, 5/240/B5, 4/300/B4} {
			\node[cycle] (\name) at ({\radius*cos(\pos)}, {\radius*sin(\pos)}) {$v_{\i}$};
		}
		
		\foreach \i/\name in {0/C3, 60/C2, 120/C1, 180/C6, 240/C5, 300/C4} {
			\node[grey] (\name) at ({2*\radius*cos(\i)}, {2*\radius*sin(\i)}) {};
		}
		
		\draw[->] (B2) -- (B1) node[midway, above] {$e_1$};
		\draw[->] (B2) -- (B3) node[midway, above right] {$e_2$};
		\draw[->] (B3) -- (B4) node[midway, below right] {$e_3$};
		\draw[->] (B4) -- (B5) node[midway, below] {$e_4$};
		\draw[->] (B6) -- (B5) node[midway, left] {$e_5$};
		
		\draw[->, color=black!10] (C1) -- (B1) {};
		\draw[->, color=black!10] (C2) -- (B2) {};
		\draw[->, color=black!10] (B3) -- (C3) {};
		\draw[->, color=black!10] (B4) -- (C4) {};
		\draw[->, color=black!10] (B5) -- (C5) {};
		\draw[-, color=black!10>] (C6) -- (B6) {};
		
		\node[cycle] (vr) at ({1.75*\radius*cos(130)}, {1.75*\radius*sin(130)}) {$v_\crr$};
		\node[cycle] (vl) at ({1.75*\radius*cos(170)}, {1.75*\radius*sin(170)}) {$v_\cl$};
		
		\draw[->] (vr) -- (B1) node[midway, below left] {$e_\crr$};
		\draw[->] (B6) -- (vl) node[midway, above right] {$e_\cl$};

		\draw[<-, fill=none, domain=180:360] plot ({\Rscale * \radius * cos(\x)}, {\Rscale * \radius * sin(\x)});
		\draw[fill=none, domain=0:120] plot ({\Rscale * \radius * cos(\x)}, {\Rscale * \radius * sin(\x)});
	\end{tikzpicture}
}
	\end{subfigure}%
	\begin{subfigure}[t]{0.45\textwidth}
		\centering
		\newcommand{\radius}{1.5} 
\newcommand{\Rscale}{0.5} 

\scalebox{0.7}{
	\begin{tikzpicture}[
		> = {Latex[scale=1.0]}, 
		line width=1pt,
		every node/.style={}
		]
		
		\tikzstyle{cycle}=[
		fill=white, 
		draw=black,
		circle,
		inner sep=1pt,
		minimum size=0pt
		]
		
		\tikzstyle{filled}=[
		fill=black, 
		draw=black,
		circle,
		inner sep=1pt,
		outer sep=0pt,
		minimum size=5pt
		]
		
		\tikzstyle{grey}=[
		fill=black!10, 
		draw=black!10,
		circle,
		inner sep=1pt,
		outer sep=0pt,
		minimum size=5pt
		]

		\foreach \i/\pos/\name in {3/0/B3, 2/60/B2, 1/120/B1, 6/180/B6, 5/240/B5, 4/300/B4} {
			\node[cycle] (\name) at ({\radius*cos(\pos)}, {\radius*sin(\pos)}) {$v_{\i}$};
		}
		
		\foreach \i/\name in {0/C3, 60/C2, 120/C1, 180/C6, 240/C5, 300/C4} {
			\node[grey] (\name) at ({2*\radius*cos(\i)}, {2*\radius*sin(\i)}) {};
		}
		
		\draw[->] (B2) -- (B1) node[midway, above] {$e_1$};
		\draw[->] (B2) -- (B3) node[midway, above right] {$e_2$};
		\draw[->] (B3) -- (B4) node[midway, below right] {$e_3$};
		\draw[->] (B4) -- (B5) node[midway, below] {$e_4$};
		\draw[->] (B6) -- (B5) node[midway, left] {$e_5$};
		
		\draw[->, color=black!10] (C1) -- (B1) {};
		\draw[->, color=black!10] (C2) -- (B2) {};
		\draw[->, color=black!10] (B3) -- (C3) {};
		\draw[->, color=black!10] (B4) -- (C4) {};
		\draw[->, color=black!10] (B5) -- (C5) {};
		\draw[-, color=black!10>] (C6) -- (B6) {};
		
		\node[cycle] (vr) at ({1.75*\radius*cos(130)}, {1.75*\radius*sin(130)}) {$v_\crr$};
		\node[cycle] (vl) at ({1.75*\radius*cos(170)}, {1.75*\radius*sin(170)}) {$v_\cl$};
		
		\draw[->] (vr) -- (B1) node[midway, below left] {$e_\crr$};
		\draw[->] (B6) -- (vl) node[midway, above right] {$e_\cl$};

		\draw[-, fill=none, domain=180:360] plot ({\Rscale * \radius * cos(\x)}, {\Rscale * \radius * sin(\x)});
		\draw[->, fill=none, domain=0:120] plot ({\Rscale * \radius * cos(\x)}, {\Rscale * \radius * sin(\x)});
	\end{tikzpicture}
}
	\end{subfigure}
	\caption{The computation of \( \beta_e\) traverses the nodes in ascending order starting at \(v_1\) (left), while \( \tilde{\beta}_e\) is computed in descending order starting at \( v_6\) (right).}
	\label{fig: computation beta}
\end{figure}
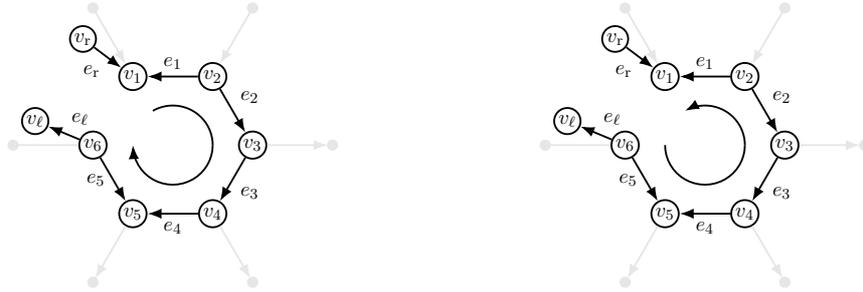

To address the question whether there exits a cut edge \( e^c\) such that \( \beta_e \ge 0\), we make two key observations: First, the modified loads \( b_v^\P\) sum up to zero, i.e.,  \(  \sum_{j=1}^{n_\cycle} b_{v_j}^\P = 0\). Second, we can retain the numbering of nodes and edges introduced by the edge \( e^c\), even if we decide to cut another edge \( \tilde{e}^c \neq e^c\).
For example, suppose we choose \( \tilde{e}^c = e_4 = (v_4,v_5)\) instead of \( (v_6, v_1)\) as the cut edge for the network in \cref{fig: numbering different cut edge}. Then we can still apply \cref{lemma: flow linear} while keeping the original numbering. In this case, the values \( \beta_{e_i}\) associated with the cut edge \( e_4\) are given by the partial sums starting at \(v_5\). Once the summation reaches the node \(  v_6\), it continuous with the node \( v_1\):
\begin{alignat*}{3}
	\beta_{e_5} & = b_{v_5}^\P, \quad & \beta_{e_6} & = b_{v_5}^\P + b_{v_6}^\P, \\
	\beta_{e_1} & =  b_{v_5}^\P + b_{v_6}^\P + b_{v_1}^\P,  
	\quad & \beta_{e_2} & =   b_{v_5}^\P + b_{v_6}^\P + b_{v_1}^\P + b_{v_2}^\P, \\
	\beta_{e_3}  & = b_{v_5}^\P + b_{v_6}^\P + b_{v_1}^\P + b_{v_2}^\P + b_{v_3}^\P, \quad & &
\end{alignat*}
To formalize this procedure, we introduce the concept of partial sums with wrapping.

\begin{figure}[ht]
	\centering
	\begin{subfigure}[t]{0.45\textwidth}
		\centering
		\newcommand{\radius}{1.5} 

\scalebox{0.7}{
	\begin{tikzpicture}[
		> = {Latex[scale=1.0]}, 
		line width=1pt,
		every node/.style={}
		]
		
		\tikzstyle{cycle}=[
		fill=white, 
		draw=black,
		circle,
		inner sep=1pt,
		minimum size=0pt
		]
		
		\tikzstyle{filled}=[
		fill=black, 
		draw=black,
		circle,
		inner sep=1pt,
		outer sep=0pt,
		minimum size=5pt
		]
		
		\tikzstyle{grey}=[
		fill=black!10, 
		draw=black!10,
		circle,
		inner sep=1pt,
		outer sep=0pt,
		minimum size=5pt
		]

		\foreach \i/\pos/\name in {3/0/B3, 2/60/B2, 1/120/B1, 6/180/B6, 5/240/B5, 4/300/B4} {
			\node[cycle] (\name) at ({\radius*cos(\pos)}, {\radius*sin(\pos)}) {$v_{\i}$};
		}
		
		\foreach \i/\name in {0/C3, 60/C2, 120/C1, 180/C6, 240/C5, 300/C4} {
			\node[grey] (\name) at ({2*\radius*cos(\i)}, {2*\radius*sin(\i)}) {};
		}
		
		\draw[->] (B2) -- (B1) node[midway, above] {$e_1$};
		\draw[->] (B2) -- (B3) node[midway, above right] {$e_2$};
		\draw[->] (B3) -- (B4) node[midway, below right] {$e_3$};
		\draw[->] (B4) -- (B5) node[midway, below] {$e_4$};
		\draw[->] (B6) -- (B5) node[midway, left] {$e_5$};
		\draw[->, dashed, color=black!30] (B6) -- (B1) node[midway, below right] {$e^c = e_6$};
		
		\draw[->, color=black!10] (C1) -- (B1) {};
		\draw[->, color=black!10] (C2) -- (B2) {};
		\draw[->, color=black!10] (B3) -- (C3) {};
		\draw[->, color=black!10] (B4) -- (C4) {};
		\draw[->, color=black!10] (B5) -- (C5) {};
		\draw[-, color=black!10>] (C6) -- (B6) {};
		
		\node[cycle] (vr) at ({1.75*\radius*cos(130)}, {1.75*\radius*sin(130)}) {$v_\crr$};
		\node[cycle] (vl) at ({1.75*\radius*cos(170)}, {1.75*\radius*sin(170)}) {$v_\cl$};
		
		\draw[->] (vr) -- (B1) node[midway, below left] {$e_\crr$};
		\draw[->] (B6) -- (vl) node[midway, above right] {$e_\cl$};

	\end{tikzpicture}
}
	\end{subfigure}%
	\begin{subfigure}[t]{0.45\textwidth}
		\centering
		\newcommand{\radius}{1.5} 

\scalebox{0.7}{
	\begin{tikzpicture}[
		> = {Latex[scale=1.0]}, 
		line width=1pt,
		every node/.style={}
		]
		
		\tikzstyle{cycle}=[
		fill=white, 
		draw=black,
		circle,
		inner sep=1pt,
		minimum size=0pt
		]
		
		\tikzstyle{filled}=[
		fill=black, 
		draw=black,
		circle,
		inner sep=1pt,
		outer sep=0pt,
		minimum size=5pt
		]
		
		\tikzstyle{grey}=[
		fill=black!10, 
		draw=black!10,
		circle,
		inner sep=1pt,
		outer sep=0pt,
		minimum size=5pt
		]

		\foreach \i/\pos/\name in {3/0/B3, 2/60/B2, 1/120/B1, 6/180/B6, 5/240/B5, 4/300/B4} {
			\node[cycle] (\name) at ({\radius*cos(\pos)}, {\radius*sin(\pos)}) {$v_{\i}$};
		}
		
		\foreach \i/\name in {0/C3, 60/C2, 120/C1, 180/C6, 240/C5, 300/C4} {
			\node[grey] (\name) at ({2*\radius*cos(\i)}, {2*\radius*sin(\i)}) {};
		}
		
		\draw[->] (B2) -- (B1) node[midway, above] {$e_1$};
		\draw[->] (B2) -- (B3) node[midway, above right] {$e_2$};
		\draw[->] (B3) -- (B4) node[midway, below right] {$e_3$};
		\draw[->, dashed, color=black!30] (B4) -- (B5) node[midway, above] {$\tilde{e}^c = e_4$};
		\draw[->] (B6) -- (B5) node[midway, left] {$e_5$};
		\draw[->] (B6) -- (B1) node[midway, above left] {$e_6$};
		
		\draw[->, color=black!10] (C1) -- (B1) {};
		\draw[->, color=black!10] (C2) -- (B2) {};
		\draw[->, color=black!10] (B3) -- (C3) {};
		\draw[->, color=black!10] (B4) -- (C4) {};
		\draw[->, color=black!10] (B5) -- (C5) {};
		\draw[->, color=black!10] (C6) -- (B6) {};
		
		\node[cycle] (vr) at ({1.75*\radius*cos(250)}, {1.75*\radius*sin(250)}) {$v_\crr$};
		\node[cycle] (vl) at ({1.75*\radius*cos(290)}, {1.75*\radius*sin(290)}) {$v_\cl$};
		
		\draw[->] (vr) -- (B5) node[midway, left] {$e_\crr$};
		\draw[->] (B4) -- (vl) node[midway, right] {$e_\cl$};
		
	\end{tikzpicture}
}
	\end{subfigure}
	\caption{A gas network with cut edge \( e^c\) and the corresponding numbering (left), and the  cut graph with cut edge \( \tilde{e}^c = e_4\) while keeping the previous numbering (right).}
	\label{fig: numbering different cut edge}
\end{figure}
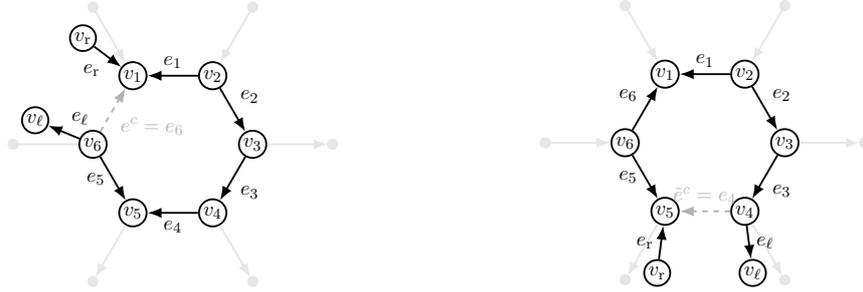

\begin{definition}[Partial Sum with Wrapping]
	Let \( (y_j)_{j=1}^n \subseteq \R\) be a finite sequence of real numbers and let \(k \in \{1, \ldots, n\}\) be a fixed index. Then the partial sum \(T_i(k) \) with wrapping starting at the element \( y_k\) is defined by
	\begin{equation*}
		T_i(k)= \begin{cases}
			\sum_{j=k}^n y_j + \sum_{j=1}^{(k+i-1)-n} y_j, & k+i-1 > n, \\
			\sum_{j=k}^{k+i-1} y_j, & k+i-1 \le n.
		\end{cases}
	\end{equation*}
\end{definition}


%

Thus, the question whether there exists a cut edge \( e^c\) such that \( \beta_e \ge 0\) becomes: Can we determine an index \( k^\ast \in \{1,\ldots,n_\cycle\}\) for the sequence \( ( b_{v_j}^\P)_{j=1}^{n_\cycle}\) such that all partial sums with wrapping \( T_i(k^\ast) \) are non-negative?
The answer gives the following result:

\begin{lemma}[Existence of Non-Negative Partial Sums with Wrapping]\label{lemma: positive partial sums}
	Let \( (y_j)_{j=1}^n \) be a finite sequence which satisfies \( \sum_{j=1}^n y_j = 0\). Then there exits an index \(k^{\ast}\) such that 
	\begin{equation*}
		T_i(k^{\ast}) \ge 0 \quad \text{for all} \quad i \in \{1, \ldots, n\}.
	\end{equation*}
\end{lemma}

\begin{proof}	
	We compute the partial sums \(S_k\) of the sequence \( (y_j)_{j=1}^n \) and determine their minimum. In particular, we have:
	\begin{equation*}
		m \coloneqq \argmin_{i\in\{1,\ldots,n\}} S_i
		\quad \text{where} \quad S_i \coloneqq \sum_{j=1}^i y_j
		\quad \Rightarrow \quad S_i \ge S_m \quad \text{for all} \quad i \in \{1, \ldots, n\}.
	\end{equation*}
	In case the minimal partial sum is non-unique, we choose \(m\) as the maximum among all minimizing indices \( \tilde{m}\).
	Then, the partial sums \( T_i(m+1)\) starting at the element \( x_{m+1}\) satisfy \( T_i(m+1) \ge 0\) for all \( i \in \{1, \ldots, n\}\). In order to prove this, we distinguish two cases:
	\begin{enumerate}[label=(\roman*)]
		\item Let \( m+i > n\) and recall that \( S_n = \sum_{j=1}^n y_j = 0\) holds. Then we have 
		\begin{align*}
			T_i(m+1) = \sum_{j=m+1}^n y_j + \sum_{j=1}^{m+i-n} y_j 
			= S_n - S_m + S_{m+i-n} = S_{m+i-n} - S_m \ge 0,					
		\end{align*}
		\item Let \( m+i \le n\). Then we have
		\begin{align*}
			T_i(m+1) = \sum_{j=m+1}^{m+i} y_j 
			= S_{m+i} - S_m \ge 0,					
		\end{align*}
	\end{enumerate}
	Thus, setting \( k^{\ast} = m + 1\) proves the claim.
\end{proof}

\begin{remark}\label{remark: partial sum reversed}
	If an index \( k^\ast\) satisfies \( T_i(k^\ast) \ge 0\), then summing in the reversed order and starting at the predecessor \( x_{k^\ast -1}\) yields non-positive partial sums. 
	
\end{remark}

Next, we transfer \cref{lemma: positive partial sums} to our application -- the existence non-negative \(\beta_e\).

\begin{lemma}\label{lemma: beta positiv}
	Let \( \G = (\V, \E)\) be a network with one cycle. Then, there exists a cut edge \( e^c \in \E_\cycle\) such that the corresponding values \( \beta_e\), as defined in \cref{lemma: flow linear}, satisfy
	\begin{equation*}
		\beta_e \ge 0 \quad \text{for all} \quad e \in \E_\cycle,
	\end{equation*}
	up to a possible change the orientation of an edge \( e \in \E_\cycle\) in the cycle.
\end{lemma}
\begin{proof}
	Fix an arbitrary edge \( \tilde{e} \in \E_\cycle\) of the cycle -- not necessarily the cut edge -- and number the nodes and edges of the cycle as described in \cref{fig: numbering}.
	Since the values \( \beta_{e_i}\) are the partial sums (with wrapping) of the modified loads \( b_v^\P\), cf. \cref{lemma: flow linear}, and since the modified loads sum up to zero, we can apply \cref{lemma: positive partial sums} to the sequence \( (b_{v_j}^\P)_{j=1}^{n_\cycle}\).
	Thus, we obtain an index \( k^\ast\) such that the corresponding partial sums with wrapping \( T_i(k^\ast)\) are all non-negative, i.e., 
	\begin{equation*}
		T_i(k^\ast) \ge 0 \quad \text{for all} \quad i = 1, \ldots, n_\cycle.
	\end{equation*}
	
	Now, we have to distinguish between the case where \( e_{k^\ast} = (v_{k^\ast-1}, v_{k^\ast}) \) and the case where \( e_{k^\ast} = (v_{k^\ast}, v_{k^\ast-1}) \).
	In the first case, we choose \( e_{k^\ast} = (v_{k^\ast-1}, v_{k^\ast})\) as the cut edge and apply \cref{lemma: flow linear}, which yields
	\begin{equation*}
		\beta_{e_i} = T_i(k^\ast) \ge 0
		\quad \text{for all} \quad i = 1, \ldots, n_\cycle.
	\end{equation*}
	
	In the second case, we have \( e_{k^\ast} = (v_{k^\ast}, v_{k^\ast-1}) \).	
	However, in this case setting the cut edge to \( e^c = e_{k^\ast}\) does not lead to the desired results because the computation of \( \beta_e\) relies on traversing the nodes of the cut cycle from \( v_\crr\) to \( v_\cl\).	
	But cutting the edge \(e_{k^\ast}\) and computing the partial sums \(T_i(k^\ast)\) corresponds to traversing the nodes of the cut cycle form \( v_\cl\) to \( v_\crr\), which means that 
	\begin{equation*}
		T_i(k^{\ast}) = \tilde{\beta}_{e_i},
	\end{equation*}
	where \( \tilde{\beta}_e\) is defined in \cref{eq: beta vl to vr}; see also \cref{fig: computation beta}. Together with \cref{remark: partial sum reversed} this implies that \( \beta_e \le 0\). 
	We resolve this problem by replacing the network \(\G\) with the network \( \tilde{\G} = (\V, \tilde{\E})\) where the edge \( e_{k^\ast}\) is flipped, i.e., \( \tilde{e}_{k^\ast} = (v_{k^\ast-1}, v_{k^\ast}) \in \tilde{\E}\). Applying the above analysis to the network \( \tilde{\G}\), we arrive at the same index \( k^\ast\) but since we flipped the edge orientation, we fall into in the first case instead.
\end{proof}

\begin{remark}
	When the second case in the proof applies, we have to continue the analysis of finding a solution to the gas flow model~\labelcref{eq: full model} with the modified network \( \tilde{\G}\).
	Once we found a solution for the network \( \tilde{\G}\), we recover a solution for the original network \( \G\) by reversing the sign of the flow on the flipped edge.
\end{remark}

With \Cref{lemma: beta positiv} at hand, we can select a cut edge such that \( \beta_e \ge 0 \) for all \( e \in \E_\cycle\), which allows us to prove that the non-linear system~\labelcref{eq: Hp and Heta} has at least one solution also for real gases with an arbitrary compressibility factor \(Z\).

\begin{lemma}\label{lemma: common root}
	Let \( \G = (\V, \E) \) be a connected, directed graph with one cycle. Further, let \( e^c\) be a \emph{suitable} edge belonging to the cycle and let \( \G^c\) be the corresponding cut graph with the boundary data given in \cref{eq: cut network boundary data}. Then, the following non-linear system, also compare \cref{eq: Hp and Heta}, has at least one solution:
	\begin{equation*}
		\H_{\eta}(\lambda, \mu)  = 0
		\quad \text{and} \quad 
		\H_p(\lambda, \mu)  = 0.
	\end{equation*}
\end{lemma}

\begin{proof}
	First, we determine a suitable cut edge \(e^c\) with \cref{lemma: beta positiv} and we set the interval \(I_\sol = [\lambda^-, \lambda^+]\) according to \cref{lemma: beta}.
	Then the function \( g(\lambda) = \H_p(\lambda, \mu_\eta(\lambda))\) is continuous on \(I_\sol\) as the root curve \(\mu_{\eta}\) is continuous on \(I_\sol\) with \cref{lemma: composition continuous} and the function \( \H_p\) is continuous on \(I_\sol \times [0,1]\) with \cref{korollar: Hp continuous}.	
	Moreover, by applying \cref{lemma: beta}, we obtain that \(g\) attains opposite sign on the boundary of \(I_\sol\).
	Finally, the intermediate value theorem yields that the function \( g\) has at least one root \( \lambda^\ast\) in \( I_\sol\). Hence, due to the definition of the curve \(\mu_\eta\) and the function \(g\), we obtain that \( (\lambda^\ast, \mu_\eta(\lambda^\ast))\) solves the non-linear system~\labelcref{eq: Hp and Heta}.	
\end{proof}

Finally, we are able to proof the \cref{thm: existence cycle}, thereby extending the result that the gas flow model~\labelcref{eq: full model} has at least one solution also for arbitrary compressibility factors. 

\begin{proof}[Proof of \cref{thm: existence cycle}]
	Choose an edge \( e^c \) belonging to the cycle of the graph \( \G\), which satisfies \cref{lemma: beta positiv}, and set it as the cut edge. Further, let \( \G^c = (\V^c, \E^c)\) be the corresponding cut graph with parameter-dependent boundary data as defined in \cref{eq: cut network boundary data}. Then, the non-linear system~\labelcref{eq: Hp and Heta} has at least one solution, see \cref{lemma: common root}. Hence, by \cite[Lemma 3.8]{Ulke2025},  the gas flow model~\labelcref{eq: full model} admits at least one solution.
\end{proof}

\section{A Numerical Example on GasLib-11}

To illustrate the practical implications of the theoretical results derived before, we present a numerical example demonstrating the influence of different compressibility factor models on the resulting steady states. In particular, we show that the choice of compressibility model affects not only the pressure profiles (cf., \cref{fig:simulationCompressibilityFactor}) but also the resulting hydrogen concentrations in the network. The simulations are carried out on the GasLib-11\footnotemark[1] benchmark network (see \cref{fig:gaslib11}), which consists of $11$ nodes, $8$ edges, two compressor stations, and one valve that can be interpreted as a friction-free pipe. The network contains a single cycle, making it well suited to highlight the effects of model choice in a setting consistent with the analytical framework developed in this work.

\footnotetext[1]{Network data can be found on \url{https://gaslib.zib.de/}}

\begin{figure}[htbp]
    \centering
    \scalebox{0.85}{
\begin{tikzpicture}[
		> = {Latex[scale=1.0]}, 
		line width=1pt,
		every node/.style={}
		]
		
		\tikzstyle{empty}=[
		fill=white, 
		draw=black,
		circle,
		inner sep=1pt,
		minimum size=2pt
		]

\pgfmathsetmacro{\L}{1.6}        
\pgfmathsetmacro{\h}{1.25}        
\pgfmathsetmacro{\s}{0.6}        

\node[empty] (v1) at (0,0) {$1$};
\node[empty] (v2) at (\L,0) {$2$};
\node[empty] (v3) at (2*\L,0) {$3$};

\node[empty] (v4) at (3*\L, \h) {$4$};
\node[empty] (v5) at (3*\L,-\h) {$5$};
\node[empty] (v6) at (3*\L,-2*\h) {$6$};
\node[empty] (v7) at (3*\L, 2*\h) {$7$};

\node[empty] (v8) at (4*\L,0) {$8$};
\node[empty] (v9) at (5*\L,0) {$9$};
\node[empty] (v10) at (6*\L, \s) {$10$};
\node[empty] (v11) at (6*\L,-\s) {$11$};

\draw[->] (v1) -- node[midway, above] {$e_1$} (v2);
\draw[->] (v2) -- node[midway,above] {$cs_1$} (v3);

\draw[->] (v3) -- node[midway,above] {$e_2$} (v4);
\draw[->] (v4) -- node[midway,right] {$e_4$} (v7);
\draw[->] (v4) -- node[midway,above] {$e_5$} (v8);

\draw[->] (v6) -- node[midway,right] {$e_3$} (v5);
\draw[->] (v5) -- node[midway,below] {$e_6$} (v8);

\draw[->] (v8) -- node[midway,above] {$cs_2$} (v9);
\draw[->] (v9) -- node[midway,above] {$e_7$} (v10);
\draw[->] (v9) -- node[midway,below] {$e_8$} (v11);

\draw[->, dashed] (v3) -- node[midway, below] {$v$} (v5);

\node (I1) at (0, -\h) {Inflow};
\node[draw=none] (I2) at (\L, -\h) {Inflow};
\node[left=0.7cm of v6, draw=none] (I3) {Inflow};

\draw[->, thick] (I1) -- (v1);
\draw[->, thick] (I2) -- (v2);
\draw[->, thick] (I3) -- (v6);

\node[right=0.7cm of v7, draw=none] (O1) {Outflow};
\node[right=0.7cm of v10, draw=none] (O2) {Outflow};
\node[right=0.7cm of v11, draw=none] (O3) {Outflow};

\draw[->, thick] (v7) -- (O1);
\draw[->, thick] (v10) -- (O2);
\draw[->, thick] (v11) -- (O3);

\end{tikzpicture}
}
    \caption{GasLib-11 with eleven nodes, eight pipes, two compressor stations and one valve.}
    \label{fig:gaslib11}
\end{figure}
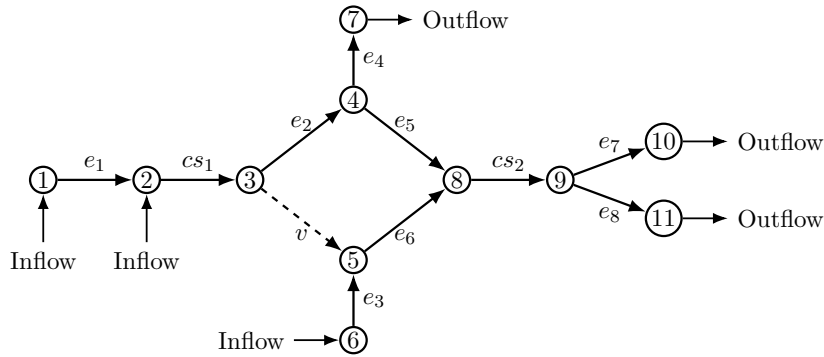

In this example, natural gas consists of $90\%$ methane, $6\%$ ethane and $4\%$ propane. For the readers' convenience we assume that pipe friction coefficient, pipe diameter and pipe length are equal for every pipe $e_i$. The network and gas parameters are listed in \cref{tab:exampleParameters}.
The gas pressure at the inflow nodes and the gas demand at the outflow nodes is given by
\begin{equation*}
    \begin{pmatrix}
        p_1 & p_2 & p_6
    \end{pmatrix}^\top 
    = \begin{pmatrix}
        60 & 58 & 63
    \end{pmatrix}^\top \text{bar}, \qquad 
    \begin{pmatrix}
        b_7 & b_{10} & b_{11} 
    \end{pmatrix}^\top
    = \begin{pmatrix}
        120 & 80 & 90
    \end{pmatrix}^\top \frac{\text{kg}}{\text{m}^2 \text{s}}.
\end{equation*}

\begin{table}[htbp]
    \centering
    \begin{tabular}{llll}
        \toprule
        Parameter  & Variable  & Value & Unit  \\
        \midrule
        pipe friction coefficient & $\lambda_\fr$ &  $0.05$ & \\
        pipe diameter & $D$ & $0.5$ & m \\
        pipe length & $L$ & $10$ & km \\
        universal gas constant & $R$ & $8.3145$ & J/(mol K)$^{-1}$ \\
        gas temperature & $T$ & $283.15$ & K \\
        critical pressure of hydrogen & $p_{c,\text{H}_2}$ & $13.15$ & bar \\
        critical pressure of natural gas & $p_{c,\text{NG}}$ & $46.01$ & bar \\
        critical temperature of hydrogen & $T_{c,\text{H}_2}$ & $33.19$ & K \\
        critical temperature of natural gas & $T_{c,\text{NG}}$ & $204.62$ & K \\
        \bottomrule
    \end{tabular}
    \caption{Parameters and coefficients for the example on GasLib-11}
    \label{tab:exampleParameters}
\end{table}

The first compressor station is switched off, i.e., $\gamma_{cs_1} = 1$, and for the second compressor station we have $\gamma_{cs_2} = 1.2$. Further, we assume that at node $1$, pure natural gas is injected, at node $2$, pure hydrogen is injected and at node $6$, a composition of $25\%$ hydrogen and $75\%$ natural gas is injected. A solution of the model \eqref{eq: full model} exists for the approximations \eqref{eq:compressibilityConstant}, \eqref{eq:compressibilityLinear} and \eqref{eq:compressibilityQuadratic} due to \cref{thm: existence cycle}. The solution is shown in \cref{fig:gaslib11Simulation}, it was computed using the \textit{AMPL} software with the \textit{ipopt} solver (see \cite{ampl2002}), while the pictures were created in \textit{MATLAB\textsuperscript{\textregistered} 2023b}. 

\begin{figure}[htbp]
    \centering
    \includegraphics[trim=20mm 0mm 20mm 0mm, clip, width=0.75\textwidth]{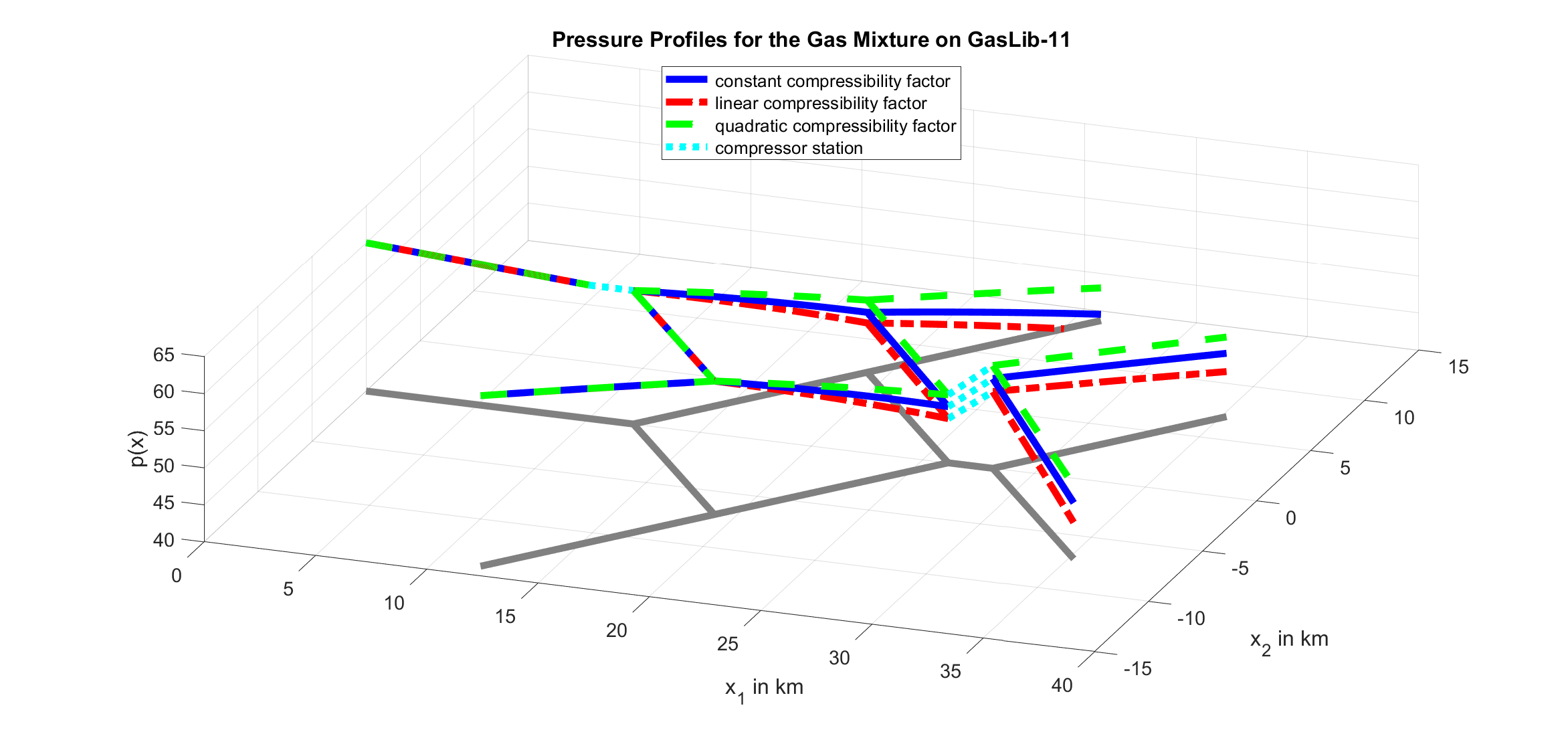}
    \caption{Simulation results for the mixing model on the GasLib--11}
    \label{fig:gaslib11Simulation}
\end{figure}

The difference in the pressure profiles is clearly visible at the outflow nodes. While the constant compressibility model \eqref{eq:compressibilityConstant} leads to $40.85$ bar, $48.54$ bar and $47.54$ bar at the outflow nodes, the linear compressibility model \eqref{eq:compressibilityLinear} leads to lower pressures of $38.64$ bar, $46.08$ bar and $44.88$ bar, and the quadratic compressibility model \eqref{eq:compressibilityQuadratic} leads to higher pressures given by $44.41$ bar, $50.73$ bar and $49.87$ bar. This observation is consistent with the simulation results shown in \cref{fig:simulationCompressibilityFactor}: the longer the transport distance - and, correspondingly, the further downstream in the network - the more pronounced the differences in the pressure profiles between the respective compressibility models become. However, in contrast to the simulation on a single pipe in \cref{fig:simulationCompressibilityFactor}, the network structure implies that the inflows, and thus the gas compositions in the pipes, also depend on the compressibility model. These results are given in \cref{tab:inflowAndComposition}, where $q_{\text{in}}$ is the inflow at the nodes $1$, $2$ and $6$, and $\eta_{\text{out}}$ is the composition at the nodes $7$, $10$ and $11$. 

\begin{table}[htbp]
    \centering
    \resizebox{\columnwidth}{!}{%
    \begin{tabular}{lccc}
        \toprule
         & constant approximation & linear approximation & quadratic approximation  \\
         \midrule 
         & & & \\[-10pt]
         \hspace{.08cm} 
         $q_{\text{in}}$ & $\begin{pmatrix}
             134.36 & 59.09 &96.55
         \end{pmatrix} \frac{\text{kg}}{\text{m}^2\text{s}}$  & \hspace{.08cm} $ \begin{pmatrix}
             147.01 & 60.30 & 82.69
         \end{pmatrix} \frac{\text{kg}}{\text{m}^2\text{s}}$ & \hspace{.08cm} $\begin{pmatrix}
             146.98 & 45.49 & 97.53
         \end{pmatrix} \frac{\text{kg}}{\text{m}^2\text{s}}$ \\
         & & & \\[-5pt]
         $\eta_{\text{out}} $& $\begin{pmatrix}
             0.4429 & 0.4159 & 0.4159
         \end{pmatrix}$ & $ \begin{pmatrix}
             0.4204 & 0.3926 & 0.3926
         \end{pmatrix}$ & $\begin{pmatrix}
             0.4348 & 0.3876 & 0.3876
         \end{pmatrix}$ \\
         & & & \\[-10pt]
         \bottomrule
    \end{tabular}%
    }
    \caption{Gas flow at the entry nodes and gas composition at the exit nodes}
    \label{tab:inflowAndComposition}
\end{table}

\section{Conclusion}

Different compressibility factor models for hydrogen-natural gas mixtures used in the literature, lead to noticeably different steady states, even with identical boundary conditions. 
This highlights that analytical results for the flow of gas mixtures should not rely on a specific functional form of the compressibility factor, but rather cover a general class of models.


In this work, we established the existence of steady states for the transport of gas mixtures in pipeline networks for a broad class of non-constant, composition-dependent compressibility factors. The analysis applies to tree-shaped networks, as well as networks containing a cycle and also includes compressor stations. 
In contrast to ideal gas, for real gases the pressure can only be represented implicitly.
Furthermore, the mixture model exhibits an additional difficulty: both pressure and composition depend on the flow direction, leading to discontinuities in the modeling. 
The existence proof is based on an implicit formulation of the steady-state pressure along pipes, combined with a continuity analysis. 
For networks with a cycle, the argument relies on cutting the cycle to obtain a tree-shaped network, solving the resulting problem with parameter-dependent boundary data, and exploiting the topological structure of the network to control the dependence of the pressure on flow and composition. This approach allows to overcome the lack of an explicit representation for the pressure and the discontinuous dependence induced by considering gas mixtures.

The results presented here provide a rigorous analytical foundation for further investigations of hydrogen-natural gas transport in pipeline networks. In particular, they open the door to optimal control problems for compressor operation and network management. Although such problems are well established for single-gas transport networks (see, e.g., \cite{BakerEtAl2022, BandaHerty2011, Gugat2024, Hante2020, HanteEtAl2017, LiersEtAl2021, SchusterHabil2026}), only numerical studies exist so far for gas mixture transport (see, e.g., \cite{Brodskyi2025, KaziSundarZlotnik2024, NayakGrundel2022}), and a comprehensive analytical treatment is still missing. The existence results derived in this paper constitute a first step towards a systematic analysis of optimal control and optimization problems for gas mixture networks with realistic real-gas effects.



%
%
\paragraph*{Acknowledgements}
Two of the authors were supported by the German Research Foundation (DFG) in the Collaborative Research Center CRC/Transregio 154, Mathematical Modelling, Simulation and Optimization using the Example of Gas Networks, Project C03, Projektnummer 239904186 (Michael Schuster) as well as under the grant GO 1920/12-1, Projektnummer 526006304 (Simone G\"ottlich). 
\printbibliography{}
\end{document}